\documentclass[letter]{article}
\usepackage[utf8]{inputenc}
\usepackage{amsfonts}
\usepackage{amsmath}
\usepackage{amsthm}
\usepackage{amssymb}
\usepackage{graphicx}
\usepackage{subcaption}
\usepackage{tikz,pgfplots}
\usepackage{epsfig}
\usepackage[normalem]{ulem}
\usepackage{multicol}

\theoremstyle{plain}
\newtheorem{theorem}{Theorem}[section]
\newtheorem{lemma}[theorem]{Lemma}
\newtheorem{proposition}[theorem]{Proposition}

\theoremstyle{definition}
\newtheorem{definition}{Definition}[section]

\theoremstyle{remark}
\newtheorem{remark}{Remark}[section]
\newtheorem{example}{Example}[section]

\newcommand{\acc}{\text{acc}}
\newcommand{\R}{{\mathbb{R}}}
\newcommand{\N}{{\mathbb{N}}}
\newcommand{\eps}{{\varepsilon}}

\title{Stability for the Training of Deep Neural Networks and Other Classifiers}

\author{Leonid Berlyand, Pierre-Emmanuel Jabin, C. Alex Safsten}

\begin{document}
	
	\maketitle
	
	\begin{abstract}
	  We examine the stability of loss-minimizing training processes that are used for deep neural networks (DNN) and other classifiers. While a classifier is optimized during training through a so-called \emph{loss function}, the performance of classifiers is usually evaluated by some measure of accuracy, such as the overall accuracy which quantifies the proportion of objects that are well classified. This leads to the guiding question of stability: does decreasing loss through training always result in increased accuracy? We formalize the notion of stability, and provide examples of instability. Our main result consists of two novel conditions on the classifier which, if either is satisfied, ensure stability of training, that is we derive tight bounds on accuracy as loss decreases. We also derive a sufficient condition for stability on the training set alone, identifying flat portions of the data manifold as potential sources of instability. The latter condition is explicitly verifiable on the training dataset. Our results do not depend on the algorithm used for training, as long as loss decreases with training.
	\end{abstract}
	
	\section{Introduction}\label{sec:introduction}
	
	Our purpose in the present article is to provide rigorous justifications to the basic training method commonly used in learning algorithms. We particularly focus on the stability of training in classification problems. Indeed, if training is unstable, it is difficult to decide when to stop. To better explain our aim, it is useful to introduce some basic notations.
	
	We are given a set of objects $S\subset\mathbb R^n$ whose elements are classified in a certain number $K$ of classes. Then we introduce a function called the \emph{exact classifier} that maps each $s\in S$ to the index $i(s)$ of its class. However, the exact classifier is typically only known on a \emph{finite} subset $T$ of $S$ called the \emph{training set}.
		

	{In practice, objects in $S$ are classified by an \emph{approximate classifier}, and the mathematical problem is to identify an optimal approximate classifier among a large set of potential candidates.}
	{The optimal approximate classifier should agree (or, at least nearly agree) with the exact classifier on $T$}. We consider here {a type of approximate classifier called} \emph{soft approximate classifiers} which may again be described in a general setting as functions
	\begin{equation}\label{eq:soft_classifier}
	\phi: s\in T\longrightarrow \phi(s)=(p_1(s),\ldots,p_K(s))\in [0,\ 1]^K,\ \mbox{with}\ p_1(s)+\ldots+p_K(s)=1.
	\end{equation}
	Such a classifier is often  interpreted as giving the probabilities that $s$ belongs to a given class: $p_i(s)$ can be described as the predicted probability that $s$ is in class $\#i$.
	In {this} framework a perfect classifier on $T$ is a function $\phi$ s.t. $p_i(s)=1$ if and only if $i=i(s)$ (and hence $p_i(s)=0$ if $i\neq i(s)$). 
	
	In practice of course, one cannot optimize among all possible functions $\phi$ and instead some sort of discretization is performed which generates {\em a parametrized family $\phi(s,\alpha)$} where the parameters $\alpha$ can be chosen in some other large dimensional space, $\alpha\in \R^\mu$ with $\mu\gg 1$ for instance. {Both $n$ and $\mu$ are typically very large numbers, and the relation between them plays a crucial role in the use of classifiers for practical problems.}
	
	Deep neural networks (DNNs) are of course one of the most popular examples of method to construct such parametrize family $\phi(s,\alpha)$. If $\phi$ is a DNN, it is a composition of several functions called \emph{layers}. Each layer is itself a composition of one linear and one nonlinear operation. In this setting, the parameters $\alpha$ are entries of matrices used for the linear operations on each layer. Figure \ref{fig:dnn} shows a diagram of a DNN.
	\begin{figure}[h]
		\centering
		\includegraphics[width=5in]{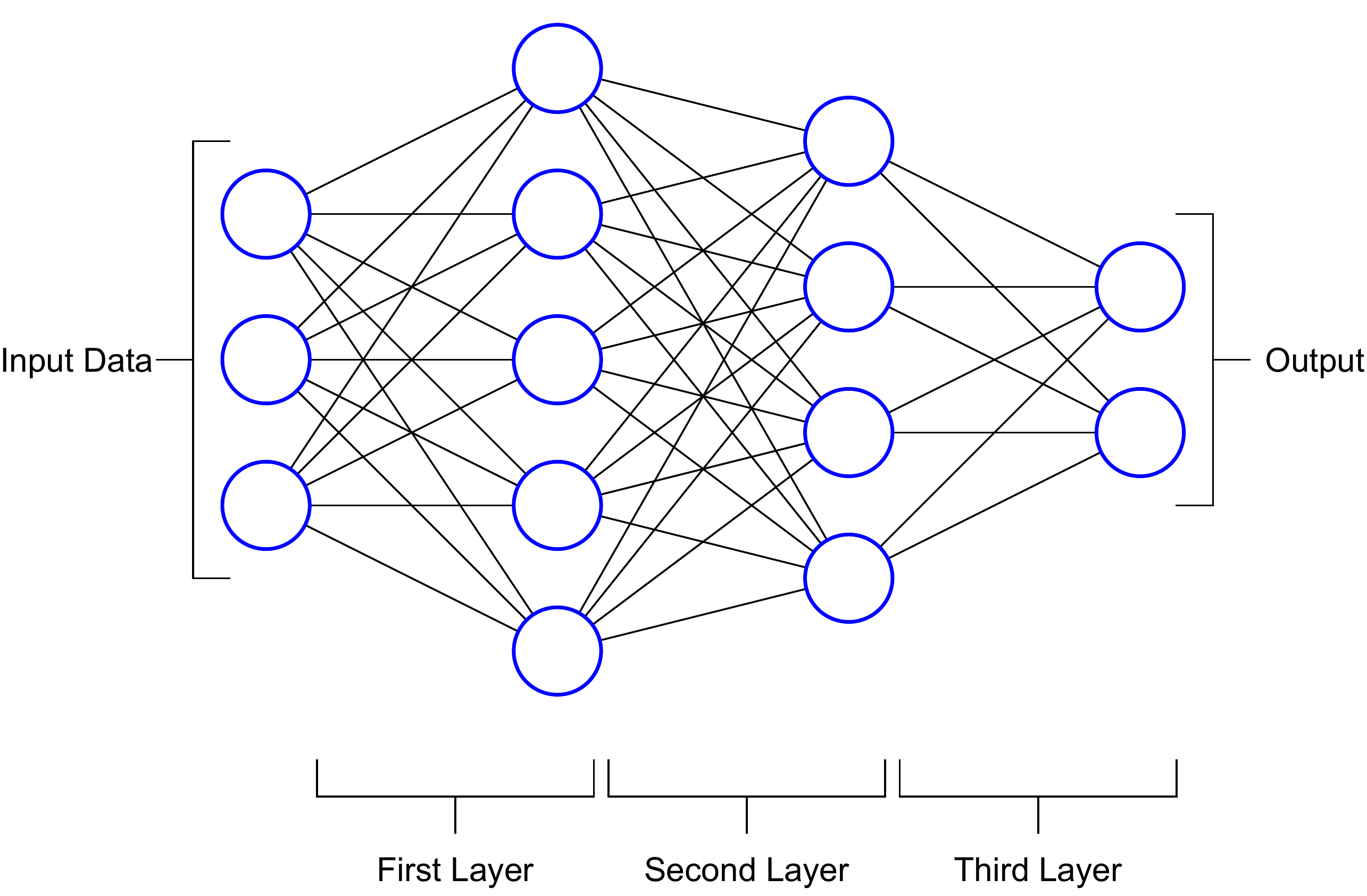}
		\caption{A diagram of a DNN.}
		\label{fig:dnn}
	\end{figure}
	
	Such deep neural networks have demonstrated their effectiveness on a variety of clustering and classifying problems, such as the {well-known} \cite{Lecun2} for handwriting recognition. To just give a few examples, one can refer to \cite{KriSutHin} for image classification, \cite{Hintonal} on speech recognition, \cite{LeungLeeFrey} for an example of applications to the {life sciences}, \cite{SutVinLe} on natural language understanding, or to \cite{LecunBengioHinton} for a general presentation of deep neural networks.
		
	\smallskip
	
	The training process consists of optimizing in $\alpha$ and  the family $\phi(s,\alpha)$ to obtain  the ``best'' choice. This naturally leads to the question of what is meant by best, which is at the heart of our investigations in this {article}. {Here, best means the highest performing classifier.} There are indeed several ways to measure the performance of classifiers: one may first consider the overall {\emph{accuracy}} which corresponds to the proportion of objects that are ``well-classified.''  {We say that $s$ is \emph{well-classified} by $\phi(\cdot,\alpha)$ if $p_{i(s)}(s,\alpha)>\max_{1\leq i\leq K,\,i\neq i(s)}p_i(s,\alpha)$,}
	that is $\phi(s,\alpha)$ gives the highest probability to the class to which $s$ {actually} belongs. Because this set will come up often in our analysis, we introduce a specific notation for this ``good'' set
	\begin{equation}\label{eq:simple_good_set_definition}
	G(\alpha)=\{s\in T\,:\ p_{i(s)}(s,\alpha)>\max_{1\leq i\leq K,\,i\neq i(s)} p_i(s,\alpha)\}, 
	\end{equation}
	which leads to the classical definition of overall accuracy
	\begin{equation}\label{eq:simple_accuracy}
	acc(\alpha)=\frac{\# G(\alpha)}{\# T}.
	\end{equation}
	For simplicity in this introduction, we assign the same weight to each object in the training set $T$. {In practice, objects in $T$ are often assigned different weights, which will be introduced below.}
        
	Other measures of performance exist and are commonly used: average accuracy where the average is taken within each class and which gives more weights to small classes, Cohen $\kappa$'s coefficient \cite{Coh1960}, etc....
	
	However, in practice training algorithms do not optimize accuracy (whether overall accuracy or some other definition) but instead try to minimize some {\em loss function}. There are compelling reasons for not using the accuracy: for example, {accuracy distinguishes between well-classified and not well-classified in a binary manner. That is, accuracy
	does not account for an object close to being well classified. Moreover, accuracy is a piecewise-constant function (taking values $0,1/\# T,2/\# T,\cdots,1$) so that its gradient is zero. For these reasons}, it cannot be maximized with gradient descent methods. Also approximating the piecewise constant accuracy function by a smooth function is computationally intensive, particularly due the high dimensionality $\mu$ of its domain.
	
	A typical example of a loss function is the so-called cross-entropy loss:
	\begin{equation}\label{eq:introducing_loss}
	\bar L(\alpha)=-\frac{1}{\#T}\,\sum_{s\in T}\log\left(p_{i(s)}(s,\alpha)\right),
	\end{equation}
	which is simply the average over the training set of the functions $-\log p_{i(s)}(s,\alpha)$. Obviously each of these functions is minimized if $p_{i(s)}(s,\alpha)=1$, {\em i.e.} if the classifier worked perfectly on the object $s$.
	
	Minimizing {\eqref{eq:introducing_loss}} simply leads to {finding} the parameters $\alpha$ such {that} on average $p_{i(s)}$ is as large (as close to $1$) as possible. Since $\bar L(\alpha)$ is smooth in $\alpha$ (at least if $\phi$ is), one may apply classical gradient algorithms, with stochastic gradient descent (SGD) being among the most popular {due to} the linear structure of $\bar L$ and the large size of many training sets. {Indeed, since \eqref{eq:introducing_loss} consists of many terms, computing the gradient of $\bar L(\alpha)$ is expensive. SGD simplifies this task by randomly selecting a batch of a few terms at each iteration of the descent algorithm, and computing the gradient of only those terms. This optimization process is referred to as \emph{training}. Though in general loss decreases through training, it need not be monotone because of e.g. effects due to stochasticity. In this work, we assume for simplicity that loss decreases monotonically, which is  approximately true in most practical problems.
	
	\smallskip
	
	The main question that we aim to answer in this article is {\em why should decreasing the loss function improve accuracy}. We start by pointing out the following  observations {which explain why the answer is not straightforward.}
	\begin{itemize}
		\item If the loss function $\bar L$  converges to $0$, then the accuracy converges to $100\%$ as in that case all predicted probability $p_{i(s)}$ converge to $1$. Nevertheless the training process is necessarily stopped at some time, before $\bar L$ reaches exactly $0$, see e.g. \cite{BerJab2018}. A first question is therefore {\em how close to $100\%$ the accuracy is when the loss function is very small}.
		\item In practice, we may not be able to reach perfect accuracy (or $0$ loss) on every training set.  This can be due to the large dimension, of the objects $s$ or of the space of parameters $\alpha$, which makes it difficult to {computationally} find a perfect minimizer even if one exists, with the usual issue of local minimizers.  Moreover, there may not even exist a perfect minimizer, due for instance to classification errors on the training set (some objects may have been assigned to the wrong class). As a consequence, a purely asymptotic comparison between loss and accuracy {as $\bar L\to 0$} is not enough, and we need to ask {\em how the loss function $\bar L$ correlates with the accuracy away from $\bar L\approx 0$}.
		\item {\em In {general}, there is no reason why decreasing $\bar L$ would increase the accuracy},  which is illustrated by the following elementary counter-example. Consider a setting with $3$ classes and an object $s$ which belongs to the first class and {such that} for the initial choice of parameter $\alpha$: $p_1(s,\alpha)=0.4,\ p_2(s,\alpha)=p_3(s,\alpha)=0.3$. We may be given a next choice of parameter $\alpha'$ {such that}: $p_1(s.\alpha')=0.45,\ p_2(s.\alpha')=0.5, p_3(s.\alpha')=0.05$. {Then the} object $s$ is well classified by the first choice of $\alpha$ and it is not well classified by the second choice $\alpha'$. Yet the loss function, which is simply $-\log p_1$ here, is obviously lower for $\alpha'$. {Thus, while loss improves, accuracy worsens.}
		\item {\em  {The previous example} raises the key issue of stability during training, which roughly speaking means that accuracy increases with training}. {Indeed,} one does not expect {accuracy to increase monotonically} during training, {and the question becomes {what conditions would guarantee that accuracy increases during training?}} {For example, one could require that}
		that the good set $G$ monotonically grows during training; in mathematical terms, that would mean that $G(\alpha)\subset G(\alpha')$ if $\alpha'$ are parameters from a later stage of the training. {However,} such a condition would be too rigid and likely counterproductive by preventing the training algorithms from reaching better classifiers. At the same time, {\em wild fluctuations in accuracy or in the good set $G(\alpha)$ would destroy any realistic hope of a successful training process, {i.e., finding a high-accuracy classifier}}.   
		\item {While the} \emph{focus of this work is on the stability during training,} another {crucial} question is the \emph{robustness} of the trained classifier, {which is the stability of identifying classes with respect to small perturbations of objects $s\in S$, in particular $s\in T$.}
		The issues of robustness and stability are  connected. Lack of stability during the training can often lead to  over-parametrization by extending the training process for too long. In turn over-parametrization typically implies  poor robustness outside of the training set. This  is connected to the Lipschitz norm of the classifier and we refer, for example, to~\cite{BalSinZou}.
		\item Since the marker of progress during training is the decrease of the loss function, stability is directly connected to how the loss function correlates with the accuracy. Per the known counterexamples, such correlation cannot always exist. Thus the {\em key question is to be able to identify which features of the dataset and of the classifier are critical {to establish such correlations and therefore} ensure stability. }
	\end{itemize}
	
	Our main contributions are to bring rigorous mathematical answers to this last question, in the context of simple deep learning algorithms {with the very popular SGD algorithm}. While part of our approach would naturally extend to other settings, it is intrinsically dependent on the  {approach} used to construct the classifier, {which is described in section \ref{sec:notations_formulation}}. More specifically, we proceed in {the following} two steps.
	\begin{itemize}
		\item[i.] We first identify conditions on the distribution of probabilities {$\left(p_1(s,\alpha),\cdots,p_K(s,\alpha)\right)$ defined in \eqref{eq:soft_classifier} for each $s\in T$
		which guarantee} that $\bar L$ correlates with accuracy. {Specifically, we show that under these conditions, loss is controlled by accuracy (vice-versa is trivial)}. At this stage such  {conditions} necessarily mix, in a non-trivial manner, the {statistical} properties of the dataset  with the properties {of the classifier (its architecture and parameters), introduced in section \ref{sec:notations_formulation}}. {Since these conditions depend on the classifier parameters which evolve with training, they cannot be verified before training starts,} 
		and they may depend on how the training proceeds.
		\item[ii.] The second step is to disentangle the previous  {conditions} to obtain separate conditions on the training set and on the neural network {architecture and parameters}. We are able to accomplish this on one of the {conditions} obtained in step $i$. 
		We provide an intuitive explanation of the main challenge here using DNN classifiers with only two classes as an example. The natural mathematical stability condition is on the distribution of the set of the differences in probability $p_{i(s)}(s,\alpha)-p_{i(s)+1}(s,\alpha)$ for each $s\in T$ which directly measures how well classified an object is (how much more likely it has to have been assigned to the correct class). Our main stability condition prevents a \emph{cluster} of very small probability differences $p_{i(s)}(s,\alpha)-p_{i(s)+1}(s,\alpha)$ in $\mathbb R$. However, a condition on the training dataset $T\subset\mathbb R^n$ is far more feasible for practical use. The challenge is finding features on the data manifold in $\mathbb R^n$ which rule out small clusters of probabilities in $\mathbb R$. The difficult part here lies in propagating the one-dimensional condition on probabilities backward through the DNN layers to a condition in $n\gg 1$ dimensions, and describing the features $T$ that may cause small clusters of $p_j(s,\alpha)$ (see Remark \ref{rmk:propagation_of_condition_B}). Specifically, if the data $T$ lies approximately on some manifold $M\subset\mathbb R^n$, flat portions of this manifold may lie parallel to the hyperplane kernel(s) of one or more of the linear maps of the DNN. In fact, only small clusters of probabilities in $\mathbb R$ prevent stability, therefore only small flat portions of the data manifold are bad for stability. The nonlinear activations of the DNN can also lead to small clusters, which can be explained as follows. Corners created through a transformation of large flat portions of $M$ by the nonlinearity of piecewise linear activation function can convert a large flat portion into multiple small flat portions (e.g., converting a straight line into a jagged line), resulting in small clusters in $\delta X$. 

	\end{itemize}
	
	Our hope is that {the present} approach and results will help develop a better understanding of why learning algorithms perform so well in many cases but still fail in other settings{. This is achieved} by providing a framework to evaluate the suitability of training sets and of neural network construction {for solving various classification problems}. 
	{Rigorous analysis of neural networks has of course already started and several approaches that are different from the present one have been introduced. We mention in particular the analysis}   
	of neural nets in terms of multiscale contractions involving wavelets and scattering transforms; see for example \cite{ChengChenMallat,Mallat,Mallat2} and \cite{CzajaLi} for  scattering transforms. {While there are a multitude of recent papers aimed to make neural net-based algorithms (also known as deep learning algorithms) faster, our goal is to help make such algorithms more stable. 
		
	  We conclude by summarizing the practical outcomes of our work:
            \begin{itemize}
            \item First, we derive and justify an explicitly verifiable conditions on the dataset that guarantee stability. We refer in particular to subsection \ref{sec:in_practice} for a discussion of how to check our conditions in practice.
            \item Our analysis characterizes how the distribution of objects in the training set and the distribution of the output of the classifier for misclassified objects affect stability of training.
              \item Finally, among many possible future directions of research, our results suggest that the introduction of \emph{multiscale} loss functions could significantly improve stability.
            \end{itemize}  
%
		
	\textbf{Acknowledgments.} The work of L. Berlyand and C. A. Safsten was supported by NSF DMREF DMS-1628411, and the Work of P.-E. Jabin was partially supported by NSF DMS Grant 161453, 1908739, NSF Grant RNMS (Ki-Net) 1107444, and LTS grant DO 14. The authors thank R. Creese and M. Potomkin for useful discussions. 
\section{Main results}\label{sec:results}
\subsection{Mathematical formulation of deep neural networks and stability \label{sec:notations_formulation}}
\subsubsection{Classifiers}\label{sec:classifier}

A parameterized family of soft classifiers $\phi(\cdot,\alpha):\mathbb R^n\to [0,1]^K$ must map objects $s$ to a list of $K$ probabilities. To accomplish this, a classifier is a composition $\phi(\cdot,\alpha)=\rho\circ X(\cdot,\alpha)$, where $X(\cdot,\alpha):\mathbb R^K\to\mathbb R^K$ and $\rho$ is the so-called \emph{softmax function} defined by
\begin{equation}\label{eq:softmax}
\rho(x)=\rho(x_1,\cdots, x_K)=\left(\frac{e^{x_1}}{\sum_{k=1}^{K}e^{x_k}},\cdots,\frac{e^{x_{K}}}{\sum_{k=1}^{K}e^{x_k}}\right).
\end{equation}
Clearly, $\rho(x)\in[0,1]^K$ and $\sum_{i=1}^K \rho_i(x)=1$, so $\rho\circ X(\cdot,\alpha)$ is a soft classifier no matter what what function $X(\cdot,\alpha)$ is used (though typically $X(\cdot,\alpha)$ is differentiable almost everywhere). The form of the softmax function means that we can write the classifier as
\begin{equation}\label{eq:classifier_with_softmax}
\phi(s,\alpha)=\rho\circ X(s,\alpha)=\left(\frac{e^{X_1(s,\alpha)}}{\sum_{k=1}^{K}e^{X_k(s,\alpha)}},\cdots,\frac{e^{X_K}}{\sum_{k=1}^{n_{K}}e^{X_k(s,\alpha)}}\right).
\end{equation} 
As in \eqref{eq:soft_classifier}, denote $\phi(s,\alpha)=(p_1(s,\alpha),\cdots,p_K(s,\alpha))$, where $p_k(s,\alpha)$ is the probability that $s$ belongs to class $k$ predicted by a classifier with parameters $\alpha$. A key property of the softmax function is that it is order preserving in the sense that if $X_i(s,\alpha)> \max_{j\neq i}X_j(s,\alpha)$, then $p_i(s)>\max_{j\neq i}p_j(s,\alpha)$. Therefore, the predicted class of $s$ can be determined by $X(s,\alpha)$. We define the key evaluation that determines whether an object is well-classified or not, namely
\begin{equation}\label{eq:defdeltaX}
\delta X(s,\alpha)=X_{i(s)}(s,\alpha)-\max_{j\neq i(s)} X_j(s,\alpha).
\end{equation}
If $\delta X(s,\alpha)>0$, then $X_{i(s)}(s,\alpha)$ is the largest component of $X(s,\alpha)$, which means that $p_{i(s)}(s,\alpha)$ is the largest probability given by $\phi(s,\alpha)$, and thus, $s$ is classified correctly. Similarly, if $\delta X(s,\alpha)<0$, $s$ is classified incorrectly.

As described in Section \ref{sec:introduction}, a classifier learns to solve the classification problem by training on a finite set $T$ where the correct classifications are known. Training is completed by minimizing a loss function which measures how far the classifier is from the exact classifier on $T$. While there are many types of loss functions, cross entropy loss introduced in \eqref{eq:introducing_loss} is very common, and it is the loss function we will consider in this work. The loss in \eqref{eq:introducing_loss} is the simple average of $-\log(p_{i(s)}(s,\alpha))$ over all $s\in T$, but there is no reason we cannot use the weighted average: 
\begin{equation}\label{eq:loss}
\bar L(\alpha)=-\sum_{s\in T}\nu(s)\log\left(p_{i(s)}(s,\alpha)\right),
\end{equation}
where $0<\nu(s)\leq 1$ and $\sum_{s\in T}\nu(s)=1$. Weights could be uniform, i.e., $\nu(s)=1/\# T$ for all $s\in T$, or weights can be non uniform if e.g. some $s\in T$ are more important than others. We can also use $\nu$ to measure the size of subset of $T$, e.g., if $A\subset T$, $\nu(A)=\sum_{s\in A}\nu(s)$. The quantity $\delta X(s,\alpha)$ defined above facilitates some convenient estimates on loss, which are shown in section \ref{sec:elementary_estimates}.

\subsubsection{Deep neural network structure}

Deep neural networks (DNNs) are a diverse set of algorithms with the classification problem being just one of many of their applications. In this article, however, we will restrict our attention to DNN classifiers. DNNs provide a useful parameterized family $X(\cdot,\alpha):\mathbb R^n\to\mathbb R^K$ which can be composed with the softmax function to form a classifier. The function $X(\cdot,\alpha)$ is a composition of several simpler functions: 
\begin{equation}\label{eq:dnn_logit}
X(\cdot,\alpha)=f_M(\cdot,\alpha_M)\circ f_{M-1}(\cdot,\alpha_{M-1})\circ\cdots\circ f_1(\cdot,\alpha_1).
\end{equation}
Each $f_k$ for $1\leq k\leq M$ is a composition of an affine transformation and a nonlinear function. The nonlinear function is called an \emph{activation function}. A typical example is the so-called \emph{rectified linear unit} (ReLU), which is defined for any integer $N\geq 0$ by $\text{ReLU}(x_1,\cdots,x_N)=\left(\max\{0,x_1\},\cdots,\max\{0,x_N\}\right).$ Another example is the componentwise absolute value, $\text{abs}(x_1,\cdots,x_N)=(|x_1|,\cdots,|x_N|)$. The affine transformation depends on many parameters (e.g., matrix elements) which are denoted together as $\alpha_k$. The collection of all DNN parameters is denoted $\alpha=(\alpha_1,\cdots,\alpha_{M-1})$.

Though we use DNN classifiers as a guiding example for this article, most results apply to classifiers of the form $\phi(\cdot,\alpha)=\rho\circ X(\cdot,\alpha)$ where $\rho$ is the softmax function, and $X$ is any family of functions parameterized by $\alpha$. In this article, we will use the term \emph{classifier} to refer to any composition $f_M\circ X(\circ,\alpha)$, while \emph{DNN classifier} refers to a classifier where $X$ has the structure of a DNN.

\subsubsection{Training, Accuracy, and Stability}\label{sec:training}
Training a DNN is the process of minimizing loss. In practice, one randomly selects a starting parameter $\alpha(0)$, and then uses an iterative minimization algorithm such as gradient descent or stochastic gradient descent to find a minimizing $\alpha$. Whatever algorithm is used, the $n^\text{th}$ iteration calculates $\alpha(n)$ using $\alpha(t)$ for $0\leq t\leq n-1$. Our results do not depend on which algorithm is used for training, but will make the essential assumption that loss decreases with training, $\bar L(\alpha(t_2))\leq\bar L(\alpha(t_2))$ for $t_2>t_1$. Throughout this article, we will abuse notation slightly by writing $\bar L(t):=\bar L(\alpha(t))$.

Accuracy is simply the proportion of well-classified elements of the training set. Using $\delta X$ and the weights $\nu(s)$, we can define a function that measures accuracy for all times $t$ during training:
\begin{equation}\label{eq:accuracy}
\acc(t)=\nu\left(\{s\in T:\delta X(s,\alpha(t))>0\}\right).
\end{equation}
We will find it useful to generalize the notion of accuracy. For instance, we may want to know how many $s\in T$ are not only well-classified, but are well-classified by some margin $\eta\geq 0$. We therefore define the \emph{good set of margin $\eta$} as
\begin{equation}\label{eq:good_set}
G_\eta(t)=\{s\in T:\delta X(s,\alpha(t))>\eta\}.
\end{equation}
Observe that $\nu(G_0(t))=\acc(t)$. For large $\eta$, the good set comprises those elements of $T$ that are exceptionally well-classified by the DNN with parameter values $\alpha(t)$. We will also consider the \emph{bad set of margin $\eta$}
\begin{equation}\label{eq:bad_set}
B_{-\eta}(t)=\{s\in T:\delta X(s,\alpha(t))\leq -\eta\}
\end{equation}
which are the elements that are misclassified with a margin of $\eta$ by the DNN with parameters $\alpha(t)$.

Stability is the idea that when, during training, accuracy becomes high enough, it remains high for all later times. Specifically, we will prove that under certain conditions, for all $\epsilon$, there exists $\delta$ and $\eta$ so that if at some time $t_0$, $\nu(G_\eta(t_0))>1-\delta$, then at all later times, $\acc(t)>1-\epsilon$.

	\subsection{Preliminary remarks and examples}\label{sec:preliminary_remarks_and_examples}

\subsubsection{Relationship between accuracy and loss}\label{sec:accuracy_and_loss}
Intuitively, accuracy and loss should be connected, i.e., as loss decreases, accuracy increases and vice versa. However, as we will see in examples below, this is not necessarily the case. Nevertheless, we can derive some elementary relations between the two. For instance, from equation \eqref{eq:loss_estimates}, we may easily derive a bound on the good set $G_\eta(t)$ via $\bar L(t_0)$ for some $\eta$ for all times $t\geq t_0$:
\begin{equation}\label{eq:good_set_and_loss}
\nu(G_\eta(t))\geq 1-\frac{\bar L(t_0)}{\log\left(1+e^{-\eta}\right)}\geq 1-2e^\eta\bar L(t_0),
\end{equation}
and in particular,
\begin{equation}\label{eq:loss_and_accuracy}
\nu(\acc(t))=\nu(G_0(t))\geq 1-\frac{\bar L(t_0)}{\log 2}.
\end{equation}
This shows if loss is sufficiently small at time $t_0$, then accuracy will be high for all later times. But this is not the same as stability; stability means that if \emph{accuracy} is sufficiently high at time $t_0$, then it will remain high for all $t>t_0$. To obtain stability from \eqref{eq:good_set_and_loss}, we somehow \emph{need to guarantee that high accuracy at time $t_0$ implies low loss at time $t_0$}. 
	
	\begin{example}\label{ex:small_cluster}
		This example will demonstrate instability in a soft classifier resulting from a small number of elements of the training set that are misclassified. Let $T$ be a training set with $1000$ elements with uniform weights, each classified into one of two classes. Suppose that at some time $t_0$, after some training, the parameters $\alpha(t_0)$ are such that most of the $\delta X(s,\alpha(t_0))$ values are positive, but a few $\delta X(s,\alpha(t_0))$ are clustered near $-0.6$. An example histogram of these $\delta X(s,\alpha(t_0))$ values is shown in Figure \ref{fig:example_1_deltaX}a. The loss $\bar L(t_0)$ accuracy can be calculated using \eqref{eq:loss} and \eqref{eq:accuracy} respectively. For the $\delta X(s,\alpha(t_0))$ values in Figure \ref{fig:example_1_deltaX}a, the loss and accuracy are is
		\begin{equation*}
		\bar L(t_0)=0.1845\quad \acc(t_0)=0.95.
		\end{equation*}
		Suppose that at some later time $t=t_0$, the $\delta X(s,\alpha(t_0))$ values are those shown in Figure \ref{fig:example_1_deltaX}b. Most $\delta X$ values have improved from $t=t_0$ to $t=t_1$, but a few have worsened. We can again calculate loss and accuracy:
		\begin{equation*}
		\bar L(t_1)=0.1772\quad \acc(t_1)=0.798.
		\end{equation*}
		Since $\bar L(t_1)<\bar L(t_0)$, this example satisfies the condition that loss must decrease during training. However, accuracy has fallen considerably. This indicates an unstable classifier. The instability arises because enough objects have sufficiently poor classifications that by improving their classification (increasing $\delta X(s,\alpha)$), training can still decrease loss if a few correctly classified objects become misclassified, decreasing accuracy.
		
		\begin{figure*}[h]
			\centering
			\begin{subfigure}[t]{0.45\textwidth}
				\centering
				\includegraphics[width=3in]{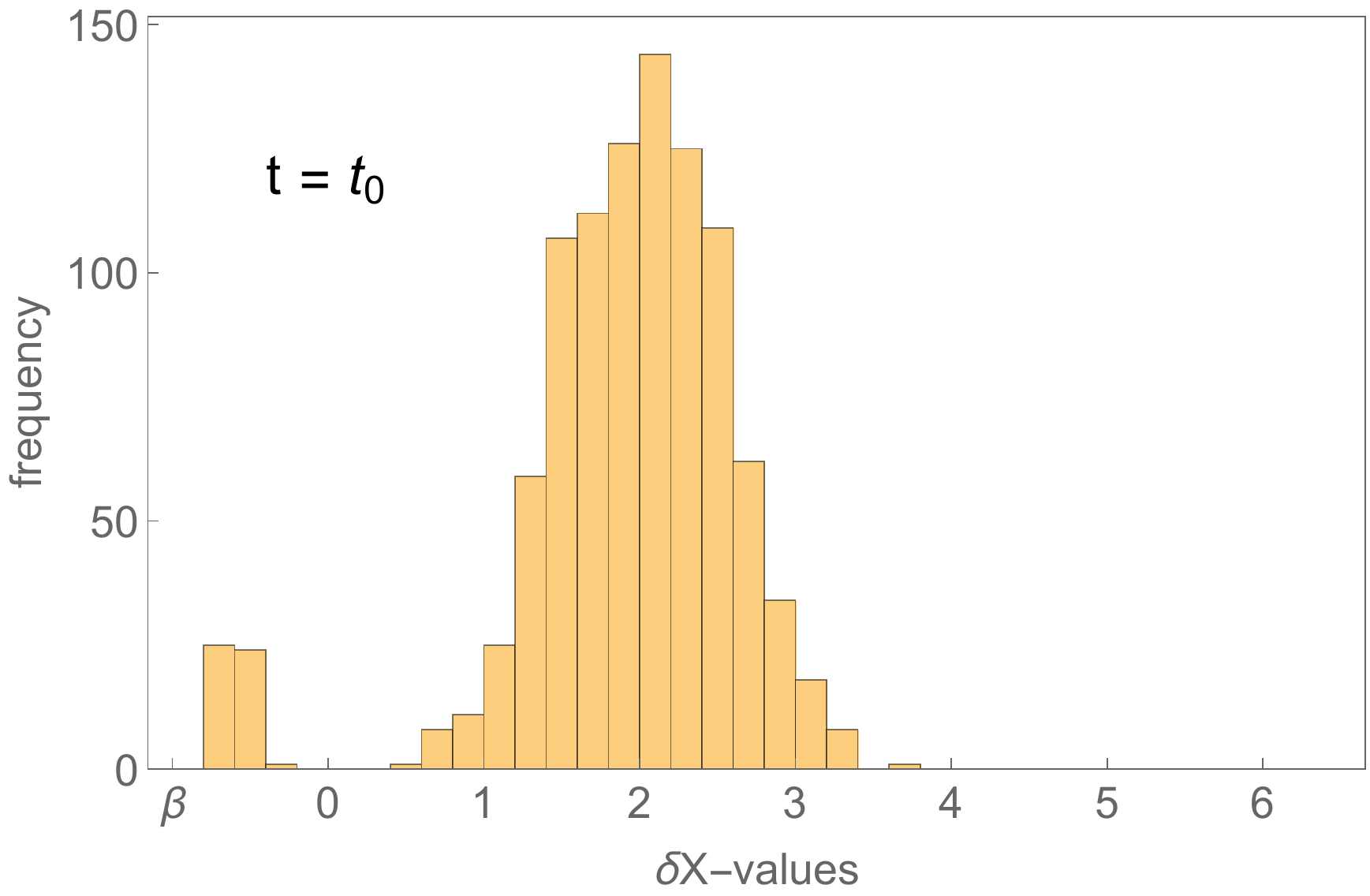}
				\caption{Histogram of $\delta X(s,\alpha(t_0))$ values}
			\end{subfigure}%
			~ 
			\begin{subfigure}[t]{0.45\textwidth}
				\centering
				\includegraphics[width=3in]{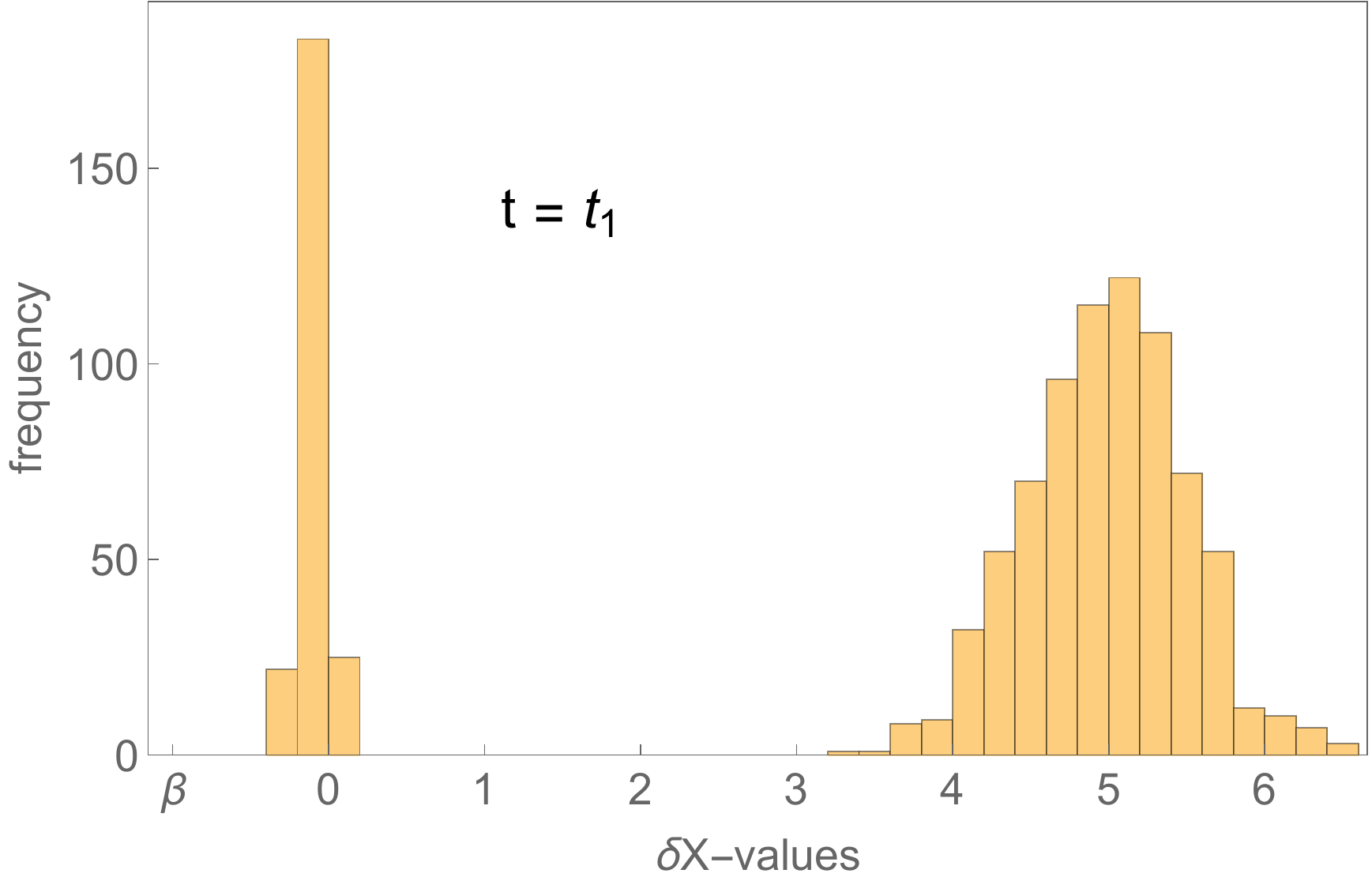}
				\caption{Histogram of $\delta X(s,\alpha(t_1))$ values}
			\end{subfigure}
			\caption{These two histograms of $\delta X(s,\alpha(t))$-values show that when $t=t_0$ (a), $95\%$ of $\delta X$-values are positive, but when $t=t_1$, only $79.8\%$ of $\delta X$ values are positive, indicating a decrease in accuracy and therefore, this DNN is unstable.}
			\label{fig:example_1_deltaX}
		\end{figure*}
	\end{example}
	\subsection{Main Results}\label{sec:main_results}
	As explained above, to establish stability, we must bound $\bar L(t_0)$ in terms of the accuracy at $t_0$. Assuming that accuracy at $t_0$ is high, we may separate the training set into a large good set $G_\eta(t_0)$ for some $\eta>0$ its small complement $G_\eta^C(t_0)$. Using \eqref{eq:loss_estimates}, we may make the following estimate, details for which are found in Section \ref{sec:appendix}.
	\begin{equation}\label{eq:first_loss_bound}
	\bar L(t_0)\leq (K-1)e^{-\eta}+\sum_{s\in G^C_\eta(t_0)}\nu(s)\log\left(1+(K-1)e^{-\delta X(s,\alpha(t_0))}\right).
	\end{equation}
	The first term in \eqref{eq:first_loss_bound} is controlled by $\eta$. If $\eta$ is even moderately large, then the second term dominates \eqref{eq:first_loss_bound}, with most of the loss coming from a few $s\in G^C_\eta(t_0)$. It is therefore sufficient to control the distribution of $\{\delta X(s,\alpha(t_0)):s\in G_\eta^C(t_0)\}$.
	There are two primary reasons why this distribution may lead to large $\bar L(t_0)$, both relating to $\beta:=\min_{s\in T}\delta X(s,\alpha(t_0))$:
	\begin{enumerate}
		\item There is a single $\delta X(s,\alpha(t_0))$ that is very large and negative, that is, $\beta\ll 0$.
		\item $\beta$ is not far from zero, but there are many $\delta X(s,\alpha(t_0))$ near $\beta$.
	\end{enumerate}
	To obtain a good bound on $\bar L(t_0)$, we must address both issues:
	\begin{enumerate}
		\item Assume that $\delta X(s,\alpha(t_0))>-1$ for all $s\in T$, i.e., the bad set $B_{-1}(t_0)$ is empty.
		\item Impose a condition that prevents $\{\delta X(s,\alpha(t)):s\in T\}$ form concentrating near $\beta$. Such conditions are called \emph{small mass conditions} since they only concern $\delta X(s,\alpha(t_0))$ for $s$ in the small mass of objects in $G^C_\eta(t_0)$. Example \ref{ex:small_cluster} illustrates this issue.
	\end{enumerate}
	{In the following subsections, we introduce two small mass conditions A and B that guarantee stability. We show that each condition leads to small loss $\bar L(t_0)$ and ultimately to stability. Condition A leads to the tightest stability result, but it must be verified for all time $t>t_0$, whereas condition B is verified only at the initial time $t=t_0$. While conditions A and B are quite precise mathematical conditions, they do not directly impose conditions on the given data set $T$, with which one usually deals in practical problems. That is why in section \ref{sec:cond_B}, we introduce the so-called \emph{no-small-isolated data clusters (NSIDC)}  condition on the data set $T$ which is sufficient for condition B, but does not depend the network parameters $\alpha(t)$. Therefore the stability result under the NSIDC condition does not depend on these parameters, unlike stability under condition B.	Even though this condition leads to a less-tight stability result than condition A, it is preferable in practice. In fact, further relaxation of the NSIDC condition for easy verification which only makes stability highly likely rather than guaranteed may be a good future direction.}
	
	\subsubsection{Condition on data and DNN for stability.}\label{sec:cond_A}
	The first small mass condition we will consider ensures that the distribution of $\delta X(s,\alpha(t_0))$ values decays very quickly near $\beta$, the minimum $\delta X(s,\alpha(t_0))$ value, so there cannot be a concentration of $\delta X$ near $\beta$, resulting in high accuracy at $t_0$ implying low loss at $t_0$. Additionally, by applying this condition at all times, not just at $t=t_0$, we can improve the estimate \eqref{eq:good_set_and_loss}.
	\begin{definition}\label{def:cond_A}
		The set $\{\delta X(s,\alpha(t)):s\in T\}$ satisfies \emph{condition A} at time $t$ if there exist constants $\Lambda \geq 1$, $m_0>0$, $\psi>0$, and $0<\phi<1$ so that for $x_1=0,\beta$ and all $x_2>x_1$,
		\begin{equation}\label{eq:cond_A}
		\begin{split}
		\nu\left(\left\{s\in T:\delta X(s,\alpha(t))<\frac{x_1+x_2}{2}\right\}\right)-m_0&\leq \Lambda \nu\left(\left\{s\in T:\delta X(s,\alpha(t))<x_1\right\}\right)^\phi\\
		&\hspace{-20pt} +\Lambda \nu\left(\left\{s\in T:\delta X(s,\alpha(t_0))<x_2\right\}\right)^{\psi+1}
		\end{split}
		\end{equation}
	\end{definition}
	How precisely this condition limits small clusters is presented in Section \ref{sec:condition_derivations}, but a brief description of the role of each constant is given here:
	\begin{itemize}
		\item $\Lambda$, the most important constant in \eqref{eq:cond_A}, controls the degree to which $\delta X$-values can concentrate near $\beta$. In particular, smaller $\Lambda$ is, the less concentrated the $\delta X$ values are, leading to a better stability result.
		\item $\psi$ controls how quickly the distribution of $\delta X(s,\alpha(t))$ values decays to zero near $\beta$. The faster the decay, the larger $\psi$ may be, and the better for stability. Calculations are greatly simplified if $\psi=1$.
		\item $\phi$ controls how fast the $\delta X$ distribution decays near $0$, leading to a bound on the size of $G_{\eta^\ast}(t)$ for $t>t_0$ and $\eta^\ast>0$. A simple calculation shows that for \eqref{eq:cond_A} to hold, we require $0<\phi<1$. 
		\item $m_0$ accounts for the fact that the training data is a discrete set. In \eqref{eq:cond_A}, if $x_1=\beta$ and $x_2=\beta+\eps$ for some small $\eps$, then it is possible that 
		\begin{equation}
			\left\{s\in T:\delta X(s,\alpha(t))<\frac{x_1+x_2}{2}\right\}=\left\{s\in T:\delta X(s,\alpha(t))<x_2\right\},
		\end{equation}
		with both sets containing a single object $s_0$. Furthermore,
		\begin{equation}
		\{s\in T:\delta X(s,\alpha(t))<x_1\}=\emptyset.
		\end{equation} 
		Thus, taking $\psi=1$, we would require $\Lambda\geq 1/\nu(s_0)\gg 1$ to satisfy \eqref{eq:cond_A}. Such a large $\Lambda$ means that the stability result will be rather weak. In this sense, \eqref{eq:cond_A} identifies the individual point $\{s_0\}$ as a very small cluster. Subtracting $m_0$ from the left side of the inequality allows the inequality to hold with much smaller $\Lambda$, accommodating a discrete dataset. Typically, we choose $m_0$ equal to about the mass of 5-10 elements of the training set. Though the presence of a small $m_0$ does not substantially affect stability, it does complicate the proofs. Therefore, in Theorem \ref{thm:stabilityA}, we assume $m_0=0$, which corresponds to the limit $\# T\to\infty$.
	\end{itemize}
	
	With condition A in hand, we can state the first stability result. Proofs and supporting lemmas are left for Section \ref{sec:proofs}.
	
	\begin{theorem}\label{thm:stabilityA}
		Suppose that $T\subset\mathbb R^n$ with weights $\nu(s)$ is a training set for a classifier such that $\{\delta X(s,\alpha(t)):s\in T\}$ which satisfies condition A for some constants $\Lambda \geq 1$, $\psi=1$, $0<\phi<1$, and $m_0=0$. for all $t\geq t_0$. Then for every $\varepsilon>0$ there exist $\delta(\Lambda ,\varepsilon),\eta(\Lambda ,\varepsilon)>0$ such that if good and bad sets at $t=t_0$ satisfy
		\begin{equation}\label{good_set_big_and_bad_set_empty_1}
		\nu(G_\eta(t_0))>1-\delta\quad\text{and}\quad B_{-1}(t_0)=\emptyset,
		\end{equation}
		then for all $t\geq t_0$,
		\begin{equation}\label{eq:stability_1}
		\text{acc}(t)=\nu(G_0(t))\geq 1-\varepsilon.
		\end{equation}
		and
		\begin{equation}\label{eq:stability_1_generalized}
		\nu(G_{\eta^\ast}(t))>1-\varepsilon-\left(3/4\right)^{-\frac{\log(3\Lambda \bar L(t_0))}{2\eta^\ast}}
		\end{equation}
		for all $\eta^\ast$: $0<\eta^\ast<-\log(3\Lambda \bar L(t_0))$.
	\end{theorem}
	
	\begin{remark}
		The conclusion of theorem \ref{thm:stabilityA} depends on the hypothesis that condition A holds. Short of brute force calculation on a case-by-case basis, there is at present no way to determine whether condition A holds for a given training set $T$ and parameter values $\alpha(t_0)$, and for which constants $\Lambda$, $\psi$, $\phi$, and $m_0$ it might hold. Furthermore, Theorem \ref{thm:stabilityA} does not control the dynamics of training with sufficient precision to guarantee that if condition A holds at $t_0$ then it will also hold for all $t>t_0$. Therefore, we have to further strengthen the hypotheses by insisting that condition A holds for all $t>t_0$.
	\end{remark}

	\begin{remark}
		Though a more general version of Theorem \ref{thm:stabilityA} can be proved for $\psi>0$ and $m_0\geq 0$, both the statement and proof are much more tractable with $\psi=1$ and $m_0=0$.
	\end{remark}
	
	\begin{remark}\label{rmk:explicit_constants_1}
		For Theorem \ref{thm:stabilityA}, it is sufficient to choose
	\begin{equation}\label{eq:explicit_constants_1}
		\delta=\frac{1}{2\Lambda }\quad\text{and}\quad\eta=\max\left\{1,\left(\frac{1}{\phi}\log\left(\left(\frac{10(K-1)}{\log 2}\right)^\phi\frac{3\Lambda }{\varepsilon}\right)\right)^2,\log\left(\frac{30 \Lambda (K-1)}{\log 2}\right)^2\right\}.
	\end{equation}
	\end{remark}
%
%
%
	\subsubsection{Doubling conditions on data independent of DNN to ensure stability}\label{sec:cond_B}

	Our second condition will also limit the number of $\delta X$ values that may cluster near $\beta$. It does this by ensuring that if a few $\delta X(s,\alpha(t_0))$ are clustered near $\beta$, then there must be more $\delta X$ values that are larger than the $\delta X$ values in the cluster. This means that the cluster near $\beta$ is not isolated, that is not all of the $s\in G^C_\eta(t_0)$ can have $\delta X(s,\alpha(t_0))$ near $\beta$.  Observe that this condition is on both the data and the DNN.  In this section we  show this second  condition follows  from a so-called doubling condition on the dataset only, which is obviously advantageous in applications
	
	Before introducing rigorously the so-called doubling condition on data that ensures stability of training, we provide its heuristic motivation. Consider data points in $\mathbb R^n$ and a bounded domain $W\subset\mathbb R^n$ containing some data points. Rescale $W$ by a factor of $2$ to obtain the ``doubled'' domain $2W$. If the number of data points in $2W$ is the same as in $W$, then the data points in $W$ form a small cluster confined in $W$. In contrast, if $2W$ contains more points than $W$, then the the cluster of points in $W$ is not isolated, which can be formulated as the following no-small-cluster (NSC) condition:
	\begin{equation}\label{eq:heuristic_doubling_condition}
	\text{\# of points in $2W$}\geq (1+\sigma)\times(\text{\# of points in $W$)~~~~~for some $\sigma>0$}
	\end{equation}
	See Figure \ref{fig:doubling_condition} for an illustration of how the doubling condition detects clusters.	
	\begin{figure}
		\centering
		\begin{subfigure}{.5\textwidth}
			\centering
			\includegraphics[width=.8\linewidth]{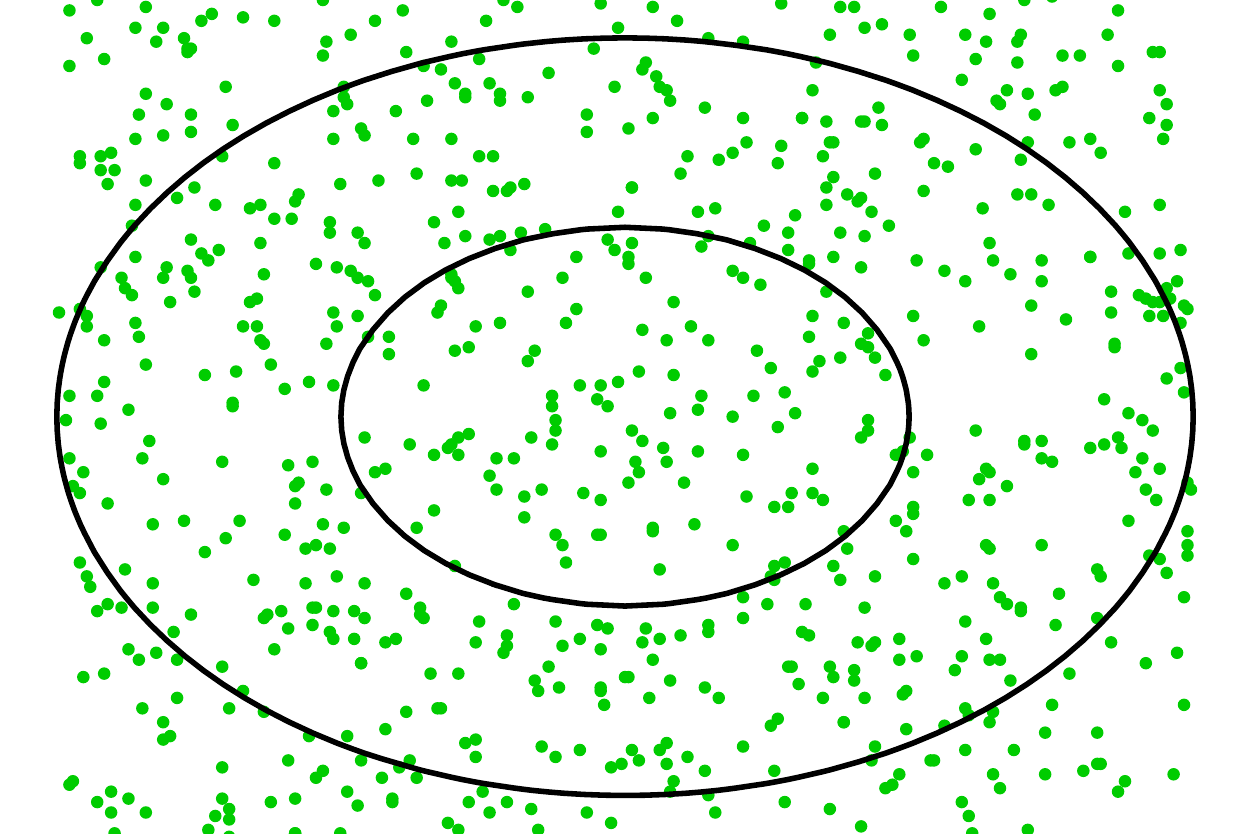}
			\caption{The big ellipse contains more data than the small ellipse so the data in the small ellipse does not comprise a cluster.}
			\label{fig:good_doubled_domain}
		\end{subfigure}%
		\begin{subfigure}{.5\textwidth}
			\centering
			\includegraphics[width=.8\linewidth]{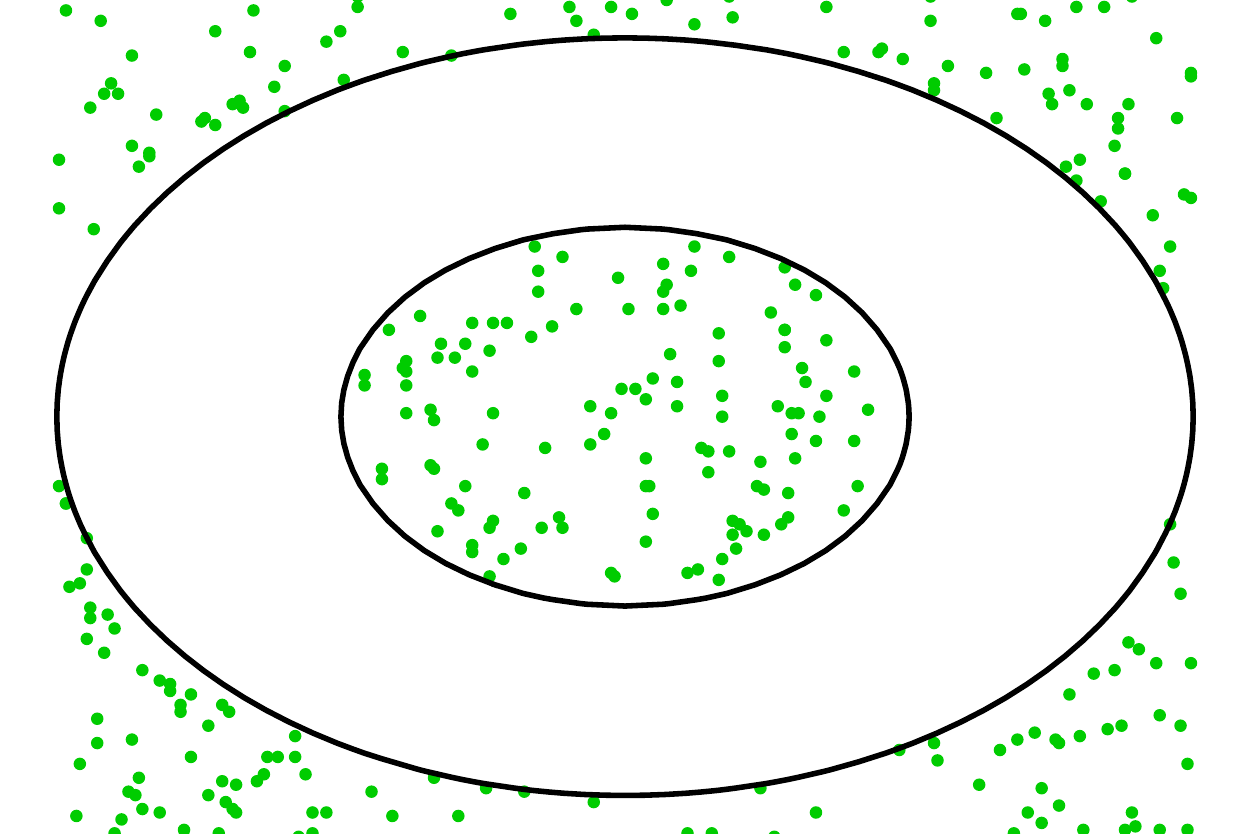}
			\caption{The big ellipse does not contain more data than the small ellipse, so the data in the small ellipse is a cluster.}
			\label{fig:bad_doubled_domain}
		\end{subfigure}
		\caption{An illustration of the idea of a doubling condition.}
		\label{fig:doubling_condition}
	\end{figure}
	This condition will be applied to $\{\delta X(s,\alpha(t_0)):s\in T\}\subset\mathbb R$ and for the data points in $T\subset R^n$. It is more practical to check this condition on the dataset which is why we here we explain the underlying heuristics for the condition on the dataset. Consider a 2D ellipse translated along a straight line to obtain an infinite cylinder with elliptical cross section. Next, introduce the truncated cylinder $E$ by intersecting it with a half-space in 3D, (e.g., all $(x,y,z)$ where $z\geq 0$). Then the toy version of the NSC condition for $E$ is
	\begin{equation}\label{eq:toy_cond}
	\nu(\kappa\varepsilon E)\geq (1+\sigma)\nu(\varepsilon E),
	\end{equation}
	where doubling has been replaced with an arbitrary rescaling parameter $\kappa>1$. Finally, $\varepsilon<\varepsilon_0$ ensures that the domain $\varepsilon E$ is sufficiently small because some features of the data manifold are seen only at small scales, and allowing large scales which rule out datasets which are fine for stability. For example, a DNN may map a flat portion of the data manifold to a point due to degeneracy of its linear maps, leading to a small cluster. Therefore small $\varepsilon$ is needed to resolve non-flatness on small scales when it may appear as flatness on a large scale, e.g., a sine curve from far away looks flat.
	
	We will present two doubling conditions which lead to stability. The first is a doubling condition on the set of $\delta X$ values in $\mathbb R$. The other is a doubling condition on the training data $T\subset\mathbb R^n$.
	 Both of the  following two definitions  provide precise  formulation of the doubling condition  \eqref{eq:heuristic_doubling_condition} and its generalization \eqref{eq:toy_cond}}.	\begin{definition}\label{def:cond_B}
		The set $\{\delta X(s,\alpha(t_0)):s\in T\}$ satisfies \emph{condition B} at time $t_0$ if there exists a mass $m_0$ and constants $\kappa>1$, and $\delta,\sigma>0$ such that for all $x_0$, there exists $I_0=I(x_0)$ and for all intervals $I\subset\mathbb R$ around $x_0$ with $|I|<I_0$,
		\begin{equation}\label{eq:cond_B}
		\nu\left(\{s\in T:\delta X(s,\alpha(t_0))\in\kappa I\}\right)\geq\min\Big\{\delta,\;\max\big\{m_0,\;(1+\sigma)\,\nu\left(\{s\in T:\delta X(s,\alpha(t_0))\in I\}\right)\big\}\Big\},
		\end{equation}
		where $\kappa I$ is the interval with the same center as $I$ but whose width is multiplied by $\kappa$.
	\end{definition}
Condition B comes with a notable advantage: it is a consequence of a similar condition on the training set $T$ called the \emph{no small data clusters condition}. For this, we define truncated slabs $Sl$ as the intersection of a slab in some direction $u$ and centered at $s$ with a finite number $k$ of linear constraints
\begin{equation}
Sl=Sl(Q,s,v_1,t_1,\ldots,v_k,t_k)=\left\{x,\ |(x-s)\cdot u|\leq 1\ \mbox{and}\ v_i\cdot x\leq t_i\ \forall i=1,\ldots, k\right\},\label{truncatedcylinder0}
\end{equation}
We also consider the rescaled slab $\kappa\,Sl$ which is obtained by changing the width of $Sl$ by a factor $\kappa$,
        \begin{equation}
\kappa\,Sl=\left\{x,\ |(x-s)^T\cdot u|\leq \kappa\ \mbox{and}\ v_i\cdot x\leq t_i\ \forall i=1,\ldots, k\right\}.\label{dilatedcylinder0}
        \end{equation}
We can now state our no small data clusters condition
	\begin{definition}\label{def:no_small_islated_data_cluster}
		Let $\bar\nu$ be the extension of the measure $\nu$ from $T$ to $\mathbb R^n$ by $\bar\nu(A)=\nu(A\cap T)$ for all $A\subset\mathbb R^n$. The \emph{no small isolated data clusters} condition holds if there exists $m_0>0$,  $\kappa>1$ and $\delta,\sigma>0$ so that for each truncated slab $Sl$, that is, each choice of $u\in\mathbb R^n$, $v_1,\ldots,v_k\in \R^n$, $s\in \R^n$, $t_1,\ldots, t_k\in \R$, there exists $\eps_0$ and for any $\eps<\eps_0$ then
		\begin{equation}\label{eq:no_small_isolated_data_cluster}
		\bar\nu(\kappa\,\eps Sl)\geq \min\Big\{\delta,\;\max\big\{m_0,\;(1+\sigma)\,\bar\nu(\eps\,Sl)\big\}-m_0\Big\}.
		\end{equation}
	\end{definition}
	
	We can now present our second stability theorem:
	\begin{theorem}\label{thm:stabilityB}
		Assume that the set $\{\delta X(s,\alpha(t_0)):s\in T\}$ satisfies condition B at $x_0=0$ as given by definition \eqref{def:cond_B} at some time $t_0$.
		Then for every $\varepsilon>0$ there exists a constant $C=C(K ,\kappa,\sigma)$ such that if
		\begin{equation}\label{eq:explicit_constants_2}
		m_0\leq \frac{\eps}{C},\quad \log \frac{1}{\eps}+C\leq \eta\leq I_0,\quad \delta_0\leq C\,\eps\,\eta^{\log(1+\sigma)/\log\kappa},
		\end{equation}
		  and if good and bad sets at $t=t_0$ satisfy
				\begin{equation}\label{good_set_big_and_bad_set_empty_B}
				\nu(G_\eta(t_0))>1-\delta_0\quad\text{and}\quad B_{-1}(t_0)=\emptyset,
				\end{equation}
				then for all $t\geq t_0$,
				\begin{equation}\label{eq:stability_B}
				\text{acc}(t)=\nu(G_0(t))\geq 1-\varepsilon.
				\end{equation}
				and
		\begin{equation}\label{eq:stability_generalized_2}
		\nu(G_{\eta^*}(t))\geq 1-2\,\log 2\,\eps\,e^{\eta^*},
		\end{equation}
		for all $\eta^\ast>0$.
	\end{theorem}
	{The following theorem guarantees that if the training set $T$ satisfies the no small data clusters condition, then the set
	$\{\delta X(s,\alpha(t_0)):s\in T\}$ satisfies condition B, no matter what what the values of the parameters $\alpha(t_0)$ are. Theorem \ref{thm:data_clusters_to_confidences} is useful because it allows us to establish stability based solely on training set $T$, without considering training. Theorems \ref{thm:stabilityA} and \ref{thm:stabilityB} apply to general classifiers, unlike the following Theorem \ref{thm:data_clusters_to_confidences} which only applies to classifiers consisting of a DNN with softmax as its last layer.}
	\begin{theorem}\label{thm:data_clusters_to_confidences}
		Let $T\subset\mathbb R^n$ be a training set for a DNN with weights $\nu(s)$ and whose activation functions are the absolute value function. For all $k$, $\kappa>1$, $\delta>0$ and $\sigma>0$, there exists $\kappa'>1$, $\delta$, and $\sigma'>0$ so that if the no small isolated data clusters condition \eqref{eq:no_small_isolated_data_cluster} holds on $T$ with constants $k'$, $\kappa'$, $\delta$, and $\sigma'$, then condition B given by \eqref{eq:cond_B} holds on $\{\delta X(s,\alpha(t_0)):s\in T\}$ with constants $\kappa$, $\delta$, and $\sigma$ for any $\alpha(t_0)$.
	\end{theorem}
The proof of Theorem \ref{thm:data_clusters_to_confidences} actually propagates a more general condition on the layers that is worth stating. Instead of truncated slabs, it applies to more general truncated ellipsoidal cylinders, namely
\begin{equation}
E=E(Q,s,v_1,t_1,\ldots,v_k,t_k)=\left\{x,\ (x-s)^T\,Q\,(x-s)\leq 1\ \mbox{and}\ v_i\cdot x\leq t_i\ \forall i=1,\ldots, k\right\},\label{truncatedcylinder1}
\end{equation}
for some symmetric {positive semi-definite} matrix $Q$. The dilated cylinders by some factor $\kappa$ are obtained with
        \begin{equation}
\kappa\,E=\left\{x,\ (x-s)^T\,Q\,(x-s)\leq \kappa^2\ \mbox{and}\ v_i\cdot x\leq t_i\ \forall i=1,\ldots, k\right\}.\label{dilatedcylinder1}
        \end{equation}
        Of course if $Q$ has rank $1$ then the definitions \eqref{truncatedcylinder1}-\eqref{dilatedcylinder1} exactly correspond to \eqref{truncatedcylinder0}-\eqref{dilatedcylinder0}.

In this more general context, the extended no small data clusters condition on any measure $\mu$ reads for given $r$, $k$, $m_0$, $\kappa$, $\sigma$ and $\delta$, 
        \begin{equation}
          \begin{split}
            &\forall Q\in M_d(\R^d)\ \mbox{with}\ \mbox{rank}\,Q\leq r,\; \forall v_1,\ldots,v_k\in \R^d\setminus\{0\},\;\forall s\in \R^d,\;\forall t_1,\ldots,t_k\in \R,\ \exists \eps_0\ \mbox{s.t.}\ \forall \eps\leq \eps_0,\\
            &\qquad \mu\left(\kappa\,\eps E(u,s,v_1,t_1,\ldots,v_k,t_k)\right)\geq\min\Big\{\delta,\; \max\big\{m_0,\ (1+\sigma)\,\mu\left(\eps\,E(u,s,v_1,t_1,\ldots,v_k,t_k)\right)\big\}-m_0\Big\}.
\end{split}
\label{doublingstrip1}
        \end{equation}
        The key role played by truncated ellipsoids is worth noting as they satisfy two important conditions:

	\begin{itemize}
		\item The image of a truncated ellipsoid by any linear map is another truncated ellipsoid;
		\item The inverse image of a truncated ellipsoid by the absolute value non-linear map consists exactly of two other truncated ellipsoids; and
		\item It is those properties that allow the propagation of the no small data clusters condition through the layers.
	\end{itemize}
Theorem \ref{thm:propogation_of_cond_B} applies to DNNs using absolute value as their activation function because (i) the absolute value function is finite-to-one unlike e.g., ReLU, and (ii) the absolute value function is piecewise linear unlike e.g., sigmoid. Property (i) is important because an infinite-to-one function will map large portions of the training set to a point, resulting in a small cluster. It is reasonable to conjecture that Theorem \ref{thm:data_clusters_to_confidences} extends to allow any finite-to-one piecewise linear activation function. It is even likely that a more restrictive version of the no-small-isolated data clusters condition will allow ReLU to be used in \ref{thm:data_clusters_to_confidences}. This is the subject of future work. Additionally (ii) is helpful because having piecewise linear activations means that each map $f_k$ (from \eqref{eq:dnn_logit}) passing from one DNN layer to the next is also piecewise linear, so propagating the condition described by \eqref{eq:no_small_isolated_data_cluster} through the DNN is more tractable than if each layer were fully nonlinear.
\subsubsection{Examples of applications of Condition A and B}\label{sec:application_of_conditions_to_example}

\textit{Revisiting example \ref{ex:small_cluster}.} 1Conditions $A$ and $B$ and their associated Theorems \ref{thm:stabilityA} and \ref{thm:stabilityB} guarantee stability, so why does stability fail in Example \ref{ex:small_cluster}? The answer lies in the constants found in conditions A and B. Condition $A$ is satisfied relative to constants $m_0$, $\Lambda $, $\phi$, and $\psi$, while condition B is satisfied relative to constants $\kappa$, $\sigma$, $\delta$, and $m_0$. We will determine for which constants conditions $A$ and $B$ are satisfied. For both conditions, we will consider a typical $m_0$ value of $m_0=0.003$.

Theorem \ref{thm:stabilityA} requires $\psi=1$. Taking $x_1=\beta(t_0)$, we find via brute force calculation that for the $\delta X$ values given in example $\Lambda $ must exceed $18.0$.  In the proof of Theorem \ref{thm:stabilityA}, we will see that $\delta $ and $\eta$ must be chosen so that $3\Lambda \bar L(t_0)<1$ However, in Example \ref{ex:small_cluster}, 
\begin{equation*}
3\Lambda \bar L(t_0)\geq 3\cdot 18\cdot 0.1845>1.
\end{equation*}
Therefore, though Theorem \ref{thm:stabilityA} guarantees the existence of $\delta$ small enough and $\eta$ large enough to get stability, such $\delta$ and $\eta$ for Example \ref{ex:small_cluster} will not satisfy the primary hypothesis of the theorem, that is $\nu(G_\eta(t_0))>1-\delta$.

Theorem \ref{thm:stabilityB} requires $\kappa=2$. If $I=[-1,-0.2]$, then $\nu(\{s\in T:\delta X(s,\alpha(t_0))\in I\})=\nu(\{s\in T:\delta X(s,\alpha(t_0))\in \kappa I\})=0.047$, so either $\sigma=0$, which is not allowed in condition B, or $\delta<.047$. But if $\delta<0.047$, then since $\nu(G_0(t_0))=0.95$, it is impossible to obtain $\nu(G_\eta(t_0))>1-\delta$ for any $\eta>0$.

Since Example \ref{ex:small_cluster} can only satisfy conditions A and B with constants that are either too large or too small, we cannot apply the stability theorems to it.

\medskip

To see how Theorems \ref{thm:stabilityA} and \ref{thm:stabilityB} can be applied, consider the following example.
\begin{example}
	Suppose a classifier properly classifies all elements in its training set at $t_0$. In fact, for some large $\eta$, $\nu(G_\eta(t_0)=1$. How large high will accuracy be at later times?
	
	First suppose that $T$ is a two-class training set for a classifier which satisfies the condition A for all $t\geq t_0$ for some constants $\Lambda\geq 1$, $\psi=1$, $0<\phi<1$ and $m_0=0$. Theorem \ref{thm:stabilityA} tells us that given $\varepsilon>0$, if $\eta$ is large enough, then 
	\begin{equation*}
	\acc(t)>1-\varepsilon
	\end{equation*}
	for all $t>t_0$. But how small can $\varepsilon$ be? Remark \ref{rmk:explicit_constants_1} gives a relationship between $\eta$ and $\epsilon$, and in particular, provided $\eta$ is sufficiently large, we may choose
	\begin{equation*}
	\varepsilon=\left(\frac{10}{\log 2}\right)^\phi\frac{3\Lambda }{e^{\phi\sqrt{\eta}}}.
	\end{equation*}
	Since $\eta$ is large, $\varepsilon$ is small. Therefore, at all later times, we are guaranteed high accuracy. Additionally, good sets also remain large. For example, by choosing 
	\begin{equation*}
	\eta^\ast=\frac{\log(3/4)}{2\log\epsilon}(\eta-\log(3\Lambda))
	\end{equation*}
	A simple computation with \eqref{eq:stability_1_generalized} using $\bar L(t_0)\leq e^{-\eta}$ shows that
	\begin{equation*}
	\nu(G_{\eta^\ast}(t))>1-2\varepsilon.
	\end{equation*}
	for all $t>t_0$. Therefore, $G_{\eta^\ast}(t)$ set remains large for all $t$, the price paid being that $\eta^\ast<\eta$. If $\eta$ is very large, we can in fact note that $\eta^\ast\sim \sqrt{\eta}$.

        Alternatively, we may choose
	\begin{equation*}
	\eta^\ast=\frac{\log{3/4}}{2\log 2}(\eta-\log(3\Lambda))
	\end{equation*}
	gives
	\begin{equation*}
	\nu(G_{\eta^\ast}(t))\geq\frac{1}{2}-\varepsilon,
	\end{equation*}
	meaning the median of the distribution of $\delta X$ values is greater than $\eta^\ast$ with now $\eta^\ast\sim\eta$.
	
	Similarly, if condition B is satisfied at $t=t_0$ for some constants $\delta$, $\kappa$, $\sigma$ and $m_0$, instead of condition A, then we can apply theorem \ref{thm:stabilityB}. Using \eqref{eq:explicit_constants_2}, we conclude that letting 
	\begin{equation*}
	\varepsilon=\max\{m_0C,e^{C-\eta}\}
	\end{equation*}
	we have
	\begin{equation*}
	\acc(t)>1-\varepsilon.
	\end{equation*}
	Since $m_0$ is small and $\eta$ is large, $\varepsilon$ is also small, so accuracy remains high. By letting $\eta^\ast=\log(2/\log 2)$, \eqref{eq:stability_generalized_2}, we have
	\begin{equation*}
	\nu(G_{\eta^\ast}(t))>1-2\varepsilon,
	\end{equation*}
	so this good set also remains large but with a significantly worse $\eta^\ast$ than for condition A.

        Finally, if $\eta^\ast=-\log(4\varepsilon\log 2)$, then
	\begin{equation*}
	\nu(G_{\eta^\ast}(t))>1/2,
	\end{equation*}
	so again the median of the $\delta X$ distribution is greater than $\eta^\ast$. Nevertheless, here we still have that $\eta^\ast\sim\eta$. 
\end{example}

\subsection{{How to verify condition A and the no-small-isolated data clusters condition for a given dataset}}\label{sec:in_practice}

In this section, we discuss how to verify conditions A and the no-small-isolated data clusters condition to ensure stability of training algorithms in a real-world setting. For completeness, we review notations:
\begin{itemize}
	\item We consider classifiers of the form $\phi(\cdot,\alpha(t))=\rho\circ X(\cdot,\alpha(t))$ where $\rho$ is the softmax function and $X(\cdot,\alpha(t)):\mathbb R^n\to\mathbb R^K$ depends on parameters $\alpha(t)$ where $t$ is the present iteration of the training process.
	\item $T$ is a training set for the classifier containing objects $s$. For each $s$, $i(s)$ is the index of the correct class of $s$.
	\item Each object $s$ has a positive weight $\nu(s)$ with $\sum_{s\in T}\nu(s)=1$.
	\item $\delta X(s,\alpha(t))=X_{i(s)}(s,\alpha(t))-\max_{j\neq i(s)}X_j(s,\alpha(t))$. 
\end{itemize}
Now we will explain how to verify condition A and the NSIDC condition, and how to use them to guarantee stability.
\begin{itemize}
	\item \textbf{Condition A.} Train the classifier until a time $t_0$ when a reasonable degree of accuracy is achieved. Calculate the values of $\delta X(s,\alpha(t_0))$ for each $s\in T$. One way to do this is to choose the typical values $\psi=1$, $\phi=1/2$, and $m_0\approx0.001$ so that the only constant to solve for in \eqref{eq:cond_A} is $\Lambda$. To this end, make the observation:
	\begin{equation}\label{eq:finding_Lambda}
	\begin{array}{cc}\text{minimal $\Lambda$ such that}\\ \text{\eqref{eq:cond_A} is satisfied}\end{array}=\max_{x_1\in\{\beta,0\},x_2>x_1}\frac{\nu\left(\left\{s\in T:\delta X(s,\alpha(t_0))<\frac{x_1+x_2}{2}\right\}\right)-m_0}{\nu\left(\left\{s\in T:\delta X(s,\alpha(t_0))<x_1\right\}\right)^\phi+\nu\left(\left\{s\in T:\delta X(s,\alpha(t_0))<x_2\right\}\right)^\psi}.
	\end{equation}
	Therefore, finding the optimal $\Lambda$ is a matter of solving a maximization problem. For fixed $x_1$, the right hand side of \eqref{eq:finding_Lambda} is a piecewise constant function with discontinuities at $\delta X(s,\alpha(t_0))$ for $s\in T$. Thus, the maximization problem is solved by sampling the right hand side of \eqref{eq:finding_Lambda} at $x_1=0$ and $x_1=\beta$, and $x_2=\delta X(s,\alpha(t_0))$ for all $s\in T$ and then finding the maximum of the resulting list of samples.
	
	With condition A satisfied for some known constants, one may apply Theorem \ref{thm:stabilityA}. However, as seen with Example \ref{ex:small_cluster}, if $\Lambda$ is too large, Theorem \ref{thm:stabilityA} may still require $G_\eta(t_0)$ larger than it actually is. It may be possible to decrease $\Lambda$ by repeating the maximization process with smaller $\phi$ or larger $m_0$. If an acceptable $\Lambda$ is found, \ref{thm:stabilityA} guarantees stability. If not, one may need to train longer to find an acceptable $\Lambda$.

	\item \textbf{NSIDC condition.} Assume that the classifier is a DNN with the absolute value function as its activation function. By verifying the no-small-isolated data clusters condition, we are guaranteed that condition B holds for all time. Therefore, we need only verify once that the no-small-isolated data clusters condition holds, which is its principal advantage over condition A. To verify this condition, choose constants $\delta=0.01$, $\kappa'=2$, and $m_0\approx 0.001$. Let $P$ be the set of all {truncated ellipsoidal cylinders} $E$ in $\mathbb R^n$ such that $\nu(\kappa\,E\cap T)\leq\delta$ and $\nu(E\cap T)>m_0$. We are left with finding the constant $\sigma'$ which satisfies \eqref{eq:no_small_isolated_data_cluster} for all {truncated ellipsoidal cylinders} in $P$:
	\begin{equation}\label{eq:finding_sigma_prime}
	\text{maximal $\sigma'$ such that \eqref{eq:no_small_isolated_data_cluster} is satisfied}=\min_{E\in P}\frac{\nu(\kappa E\cap T)}{\nu(E\cap T)}-1.
	\end{equation}
	As in verifying condition A, we will solve this minimization problem by discretizing the domain of minimization, $P$, and sampling the objective function only on that discretization. It should be noted that the dimension of $P$ is {of order $n^2$}. Since $n$ (the dimension of the space containing $T$) is typically high, this means that a sufficiently fine discretization is necessarily quite large.
	
	After finding $\sigma'$, we may be sure that condition B is satisfied at every iteration of the training algorithm for known constants. Therefore, we may apply Theorem \ref{thm:stabilityB} to guarantee stability.
\end{itemize}

\section{Proof of Theorems \ref{thm:stabilityA}, \ref{thm:stabilityB} and \ref{thm:data_clusters_to_confidences}}\label{sec:proofs}

\subsection{Elementary estimates}\label{sec:elementary_estimates}

Here, we will show the details and derivations of many several simple equations, inequalities, and some technical lemmas.

\smallskip

\textbf{Estimates for loss.} As mentioned in Section \ref{sec:accuracy_and_loss}, the quantity $\delta X(s,\alpha)$ facilitates convenient estimates for loss. To start, \eqref{eq:classifier_with_softmax} gives
\begin{equation}\label{eq:transforming_softmax}
p_{i(s)}(s,\alpha)=\frac{e^{X_{i(s)}(s,\alpha)}}{\sum_{k=1}^K e^{X_k(s,\alpha)}}=\frac{1}{\sum_{k=1}^K e^{X_k(s,\alpha)-X_{i(s)}(s,\alpha)}}=\frac{1}{1+\sum_{k\neq i(s)} e^{X_k(s,\alpha)-X_{i(s)}(s,\alpha)}}.
\end{equation}
For each $k\neq i(s)$, $X_k(s,\alpha)-X_{i(s)}(s,\alpha)\leq\max_{k\neq i(s)} X_k(s,\alpha)- X_{i(s)}(s,\alpha)=-\delta X(s,\alpha)$. Using \eqref{eq:defdeltaX} and \eqref{eq:transforming_softmax}, we obtain estimates on $p_{i(s)}(s,\alpha)$:
\begin{equation}\label{eq:probability_estimates}
\frac{1}{1+(K-1)e^{-\delta X(s,\alpha)}}\leq p_{i(s)}(s,\alpha)\leq\frac{1}{1+e^{-\delta X(s,\alpha)}}.
\end{equation} 
Finally, \eqref{eq:probability_estimates} gives estimates on loss:    
\begin{equation}\label{eq:loss_estimates}
\sum_{s\in T}\nu(s)\log\left(1+e^{-\delta X(s,\alpha)}\right)\leq\bar L(\alpha)\leq\sum_{s\in T}\nu(s)\log\left(1+(K-1)e^{-\delta X(s,\alpha)}\right).
\end{equation}
In particular, if there are only two classes, the inequalities in \eqref{eq:probability_estimates} and \eqref{eq:loss_estimates} are equalities.

\smallskip

\textbf{Derivations of \eqref{eq:good_set_and_loss} and \eqref{eq:loss_and_accuracy}.} Equations \eqref{eq:good_set_and_loss} and \eqref{eq:loss_and_accuracy} show how low loss leads to high accuracy. Starting from the lower bound for loss in \eqref{eq:loss_estimates}, we make the following estimates for any $t$:
\begin{align}\nonumber
\bar L(t)&\geq\sum_{s\in T}\nu(s)\log(1+e^{-\delta X(s,\alpha(t))})\\ \nonumber
&\geq\sum_{s\in G_\eta^C(t)}\nu(s)\log(1+e^{-\delta X(s,\alpha(t))})\\ \nonumber
&\geq\log(1+e^{-\eta})\sum_{s\in G_\eta^C(t)}\nu(s)\\ \nonumber
&=\nu(G_\eta^C(t))\log(1+e^{-\eta})\\
&=(1-\nu(G_\eta(t)))\log(1+e^{-\eta}). \label{eq:intermediate_good_set_estiamte}
\end{align}
Observe that $\log(1+x)\geq x/2$ for $0\leq x\leq 1$. It follows that
\begin{equation*}
\nu(G_\eta(t))\geq 1-2e^\eta\bar L(t).
\end{equation*}
With the assumption that loss is decreasing, and $t\geq t_0$, we conclude that
\begin{equation*}
\nu(G_\eta(t))\geq 1-2e^\eta\bar L(t)\geq 1-2e^\eta\bar L(t_0).
\end{equation*}
On the other hand, we can obtain an improved estimate for $\nu(\acc(t))$ by applying \eqref{eq:intermediate_good_set_estiamte} with $\eta=0$:
\begin{equation*}
\nu(\acc(t))\geq 1-\frac{\bar L(t)}{\log 2}\geq 1-\frac{\bar L(t_0)}{\log 2}.
\end{equation*}

\smallskip

\textbf{Derivation of \eqref{eq:first_loss_bound}.} Equation \eqref{eq:first_loss_bound} shows that when $G_\eta(t)$ is large, the sum \eqref{eq:loss} is dominated by a few terms that correspond to poorly classified objects. To derive \eqref{eq:first_loss_bound}, start from the upper bound in \eqref{eq:loss_estimates}, and then make the following series of estimates:
\begin{align*}
\bar L(t)&=\sum_{s\in T}\nu(s)\log\left(1+(K-1)e^{-\delta X(s,\alpha(t))}\right)\\
&=\sum_{s\in G_\eta(t)}\nu(s)\log\left(1+(K-1)e^{-\delta X(s,\alpha(t))}\right)+\sum_{s\in G_\eta^C(t)}\nu(s)\log\left(1+(K-1)e^-{\delta X(s,\alpha(t))}\right)\\
&\leq\log\left(1+(K-1)e^{-\eta}\right)\sum_{s\in G_\eta(t)}\nu(s)+\sum_{s\in G_\eta^C(t)}\nu(s)\log\left(1+(K-1)e^{-\delta X(s,\alpha(t))}\right)\\
&\leq(K-1)e^{-\eta}\nu(G_\eta(t))+\sum_{s\in G_\eta^C(t)}\nu(s)\log\left(1+(K-1)e^{-\delta X(s,\alpha(t))}\right)\\
&\leq(K-1)e^{-\eta}+\sum_{s\in G_\eta^C(t)}\nu(s)\log\left(1+(K-1)e^{-\delta X(s,\alpha(t))}\right).
\end{align*}

The following two technical lemmas will be used in later proofs.

\begin{lemma}\label{lem:double_exponential_sum_constant}
	Suppose $0<\Lambda \delta<1/2$. The inequality
	$$\sum_{k=0}^\infty e^{2^k\log(\Lambda \delta)-\frac{1+\eta}{2^{k+1}}}\leq C\,e^{-\sqrt{-2\log( \Lambda \delta)(1+\eta)}}$$
	can always be satisfied for some $C=C(\Lambda ,\varepsilon,\eta)\leq 13/4$	
\end{lemma}

The proof of Lemma \ref{lem:double_exponential_sum_constant} is essentially a long series of elementary estimates which are not very enlightening. Consequently, it is relegated to the appendix.

\begin{lemma}\label{lem:double_exponential_sum_constant_2}
	For any $p>0$ and any $\kappa\geq 2$,
	\begin{equation}\label{eq:double_exponential_sum_2}
	\sum_{0}^{p-1}\kappa^{i}\,e^{-\kappa^i}\leq 2.
	\end{equation}
\end{lemma}
\begin{proof}
	Observe that for $x\geq 2$, one trivially has that
	\[
	x^2\,e^{-x}\leq 1,
	\]
	so that
	\[
	\sum_{1}^{p-1}\kappa^{i}\,e^{-\kappa^i}\leq \sum_{1}^{p-1} \kappa^{-i}\leq \sum_{1}^{p-1} 2^{-i}\leq 1.
	\]
\end{proof}

\subsection{Upper bound on the loss function for condition A}
A key part of the proof of Theorem \ref{thm:stabilityA} is to obtain an upper bound on the Loss function at the initial time $t_0$, as given by
\begin{lemma}\label{lem:loss_bound_1}
	Suppose that $T\subset\mathbb R^n$ with measure $\nu(s)$ is a training set for a softmax DNN which satisfies condition A for some constants $\Lambda \geq 1$, $\psi=1$ and $0<\phi<1$ and for $t=t_0$. If for some $\eta>0$,
	\begin{equation}\label{eq:good_set_big_1a}
	\nu(G_\eta(t))>1-\delta
	\end{equation}
	for $\delta<1/2\Lambda $ and
	\begin{equation}\label{eq:bad_set_empty_1a}
	B_{-1}(t)=\emptyset,
	\end{equation}
	then the cross-entropy loss is bounded by: 
	\begin{equation}\label{eq:loss_bounded_1}
	\bar L(t)\leq e^{-\eta}+Ce^{-\sqrt{-2\log(\Lambda \delta)(1+\eta)}},
	\end{equation}
	where $C$ is a constant less than $13e/4$.
\end{lemma}
\begin{proof}
	For each $k=0,1,2,\cdots...$, let
	\begin{equation*}
	\eta_k=\beta+\frac{-\beta+\eta}{2^k}\;\;\text{and}\;\;I_k=\left\{s\in T:\delta X(s,\alpha(t_0))<\eta_k\right\},
	\end{equation*}
	where $\beta:=\min_{s\in T}\delta X(s,\alpha(t_0))$. Observe the relation $(\eta_k+\beta)/2=\eta_{k+1}$. For fixed $k$, apply condition A for $x_1=\beta$ and $x_2=\eta_k$:
	\begin{align*}
	\nu(I_{k+1})&=\nu\left(\left\{s\in T:\delta X(s,\alpha(t_0))<\eta_{k+1}\right\}\right)\\
	&=\nu\left(\left\{s\in T:\delta X(s,\alpha(t_0))<\frac{\eta_{k}+\beta}{2}\right\}\right)\\
	&=\nu\left(\left\{s\in T:\delta X(s,\alpha(t_0))<\frac{x_1+x_2}{2}\right\}\right).\\
        \end{align*}
This implies that        
        \begin{align*}
	\nu(I_{k+1})&=\Lambda \nu\left(\left\{s\in T:\delta X(s,\alpha(t_0))<x_1\right\}\right)^\phi+\Lambda \nu\left(\left\{s\in T:\delta X(s,\alpha(t_0))<x_2\right\}\right)^{\psi+1}\\
	&=\Lambda \nu\left(\left\{s\in T:\delta X(s,\alpha(t_0))<\beta\right\}\right)^\phi+\Lambda \nu\left(\left\{s\in T:\delta X(s,\alpha(t_0))<\eta_k\right\}\right)^{\psi+1}\\
	&=0+\Lambda \nu(I_k)^2\quad\text{by the definition of $\beta$}\\
	&=\Lambda \nu(I_k)^2
	\end{align*}
	Since $\eta_0=\eta$, \eqref{eq:good_set_big_1a} gives
	\begin{equation*}
	\nu(I_0)=\nu\left(\left\{s\in T:\delta X(s,\alpha(t))<\eta\right\}\right)=1-\nu(G_\eta)\leq\delta.
	\end{equation*}
	Therefore, by induction
	\begin{equation*}
	\nu(I_k)\leq \Lambda ^{2^k-1}\delta^{2^k}.
	\end{equation*}
	We may now simply bound loss from above. Recall that the cross-entropy loss may be bounded by
	\begin{equation*}
	\bar L(t)\leq\sum_{s\in T}\nu(s)\log\left(1+(K-1)e^{-\delta X(s,\alpha(t))}\right).
	\end{equation*}
	Since $\beta$ is the minimum $\delta X$ value, $\delta X(s,\alpha(t_0))>\beta$ for all $s\in T$, so either $s\in G_\eta(t)$ or $s\in I_k\setminus I_{k+1}$ for some $k$. Additionally, $\log(1+(K-1)e^{-x})$ is decreasing in $x$, so if $\delta X(s,\alpha(t_0))\in I_k\setminus I_{k+1}$, then $\delta X(s,\alpha(t_0))>\eta_{k+1}$. Therefore,
	\begin{equation*}
	\log\left(1+(K-1)e^{-\delta X(s,\alpha(t))}\right)\leq \log\left(1+(K-1)e^{-\eta_{k+1}}\right).
	\end{equation*} 
	Using these facts, we make the following estimate
	\allowdisplaybreaks
	\begin{align*}
	\bar L(t_0)&\leq\sum_{s\in G_\eta(t)}\nu(s)\log\left(1+(K-1)e^{-\delta X(s,\alpha(t_0))}\right)+\sum_{k=0}^\infty\sum_{s\in I_k\setminus I_{k+1}}\nu(s)\log\left(1+(K-1)e^{-\delta X(s,\alpha(t_0))}\right)\\
	&\leq\sum_{s\in G_\eta(t)}\nu(s)\log\left(1+(K-1)e^{-\eta}\right)+\sum_{k=0}^\infty\sum_{s\in I_k\setminus I_{k+1}}\nu(s)\log\left(1+(K-1)e^{-\eta_{k+1}}\right)\\
	&=\log\left(1+(K-1)e^{-\eta}\right)\sum_{s\in G_\eta(t)}\nu(s)+\sum_{k=0}^\infty\log\left(1+(K-1)e^{-\eta_{k+1}}\right)\sum_{s\in I_k\setminus I_{k+1}}\nu(s).\\
        \end{align*}
As a consequence, we have that
        \begin{align*}
	\bar L(t_0)&\leq \nu(G_\eta(t))(K-1)e^{-\eta}+\sum_{k=0}^\infty \nu(I_k\setminus I_{k+1})(K-1)e^{-\eta_{k+1}}\\
	&\leq (1-\delta) (K-1)e^{-\eta}+\sum_{k=0}^\infty \nu(I_k)(K-1)e^{-\eta_{k+1}}\\
	&\leq (K-1)e^{-\eta}+\sum_{k=0}^\infty (K-1)\Lambda ^{2^k-1}\delta^{2^k}e^{-\eta_{k+1}},\\
        \end{align*}
and
        \begin{align*}
	\bar L(t_0) &= (K-1)e^{-\eta}+(K-1)\sum_{k=0}^\infty \Lambda ^{2^k-1}\delta^{2^k}e^{-\beta-\frac{\eta-\beta}{2^{k+1}}}\\
	&\leq(K-1) e^{-\eta}+(K-1)e^{-\beta}\sum_{k=0}^\infty e^{2^k\log(\Lambda \delta)-\frac{\eta-\beta}{2^{k+1}}}.
	\end{align*}
	Since $\beta>-1$, we have
	\begin{equation*}
	\bar L(t_0)\leq (K-1) e^{-\eta}+(K-1)e\sum_{k=0}^\infty e^{2^k\log(\Lambda \delta)-\frac{\eta+1}{2^{k+1}}}.
	\end{equation*}
	By Lemma \ref{lem:double_exponential_sum_constant}, we can find a constant $C$ less than $13e/4$ such that
	\begin{equation*}
	e\sum_{k=0}^\infty e^{2^k\log(\Lambda \delta)-\frac{\eta+1}{2^{k+1}}}\leq Ce^{-\sqrt{-2\log(\Lambda \delta)(1+\eta)}}.
	\end{equation*}
	Thus,
	\begin{equation*}
	\bar L(t)\leq (K-1)\left(e^{-\eta}+Ce^{-\sqrt{-2\log(\Lambda \delta)(1+\eta)}}\right).
	\end{equation*}		
\end{proof}
\subsection{Proof of Theorem \ref{thm:stabilityA}}
With Lemma \ref{lem:loss_bound_1}, we are now ready to prove Theorem \ref{thm:stabilityA}.
	Fix $\varepsilon>0$. Let $\delta_0=1/2\Lambda $, and choose $\delta<\delta_0$. We may apply Lemma \ref{lem:loss_bound_1} to see that
	\begin{equation*}
	\bar L(t_0)\leq P(\eta,\delta,\Lambda ):= (K-1)\left(e^{-\eta}+Ce^{-\sqrt{-2\log(\Lambda \delta)(1+\eta)}}\right).
	\end{equation*}
	Observe that $\lim_{\eta\to\infty}P(\eta,\delta,\Lambda )=0$, so by choosing $\eta_0$ sufficiently large and $\eta>\eta_0$, we can make loss arbitrarily small (see Remark \ref{rmk:explicit_constants_1} for an explicit estimate on $\eta_0$). 
	

	Since loss is decreasing in time, if $\bar L(t_0)$ is small, then $\bar L(t)$ is also small for all $t>t_0$. By making $\bar L(t)$ small, we will be able bound the size of good sets from below. To start, we will show that if $\bar L(t_0)<1/3\Lambda $, then
	\begin{equation}\label{eq:principal_good_set}
	G_{-\log(3\Lambda \bar L(t_0))}(t)>1-\frac{1}{2\Lambda }
	\end{equation}
	for all $t$. Suppose, to the contrary, that $G_{-\log(3\Lambda \bar L(t_0))}(t)\leq 1-\frac{1}{2\Lambda }$ for some $t$. Then the size of the complement of the good set is bounded below:
	\begin{equation*}
	G_{-\log(3\Lambda \bar L(t_0))}^C(t)>\frac{1}{2\Lambda }
	\end{equation*} 
	Therefore,
	\begin{align*}
	\bar L(t_0)&\geq\bar L(t)\\
	&\geq\sum_{s\in T}\nu(s)\log(1+e^{-\delta X(s,\alpha(t))})\\
	&\geq\sum_{s\in G^C_{-\log(3\Lambda\bar L(t_0))}}\nu(s)\log(1+e^{-\delta X(s,\alpha(t))}),\\
        \end{align*}
        which gives
        \begin{align*}
	\bar L(t_0)&\geq\nu(G_{-\log(3\Lambda \bar L(t_0))}^C(t))\log(1+e^{\log(3\Lambda \bar L(t_0))})\\
	&>\frac{1}{2\Lambda }\log(1+3\Lambda \bar L(t_0)).
	\end{align*}
	Since $\log$ is concave down, $\log(1+x)<x\log(2)$ for $0<x<1$. Thus,
	\begin{equation*}
	\bar L(t_0)>\frac{1}{2\Lambda }3\Lambda \bar L(t_0)\log(2)=\frac{3}{2}\log(2)\bar L(t_0)>\bar L(t_0).
	\end{equation*}
	This is a contradiction, so $G_{-\log(3\Lambda \bar L(t_0))}(t)>1-\frac{1}{2\Lambda }$.
	
	We will now use \eqref{eq:principal_good_set} to bound the size of all good sets from below. Suppose that $\bar L(t_0)<1/3\Lambda $. Let $\eta_k=-2^{-k}\log(3\Lambda \bar L(t_0))$ for all $k=0,1,2,\cdots$. Clearly, there is a recurrence relation: $\eta_{k+1}=\eta_k/2$. For fixed $k$, let $x_1=0$ and $x_2=\eta_k$. Then by condition A we can estimate the size of the complement of $G_{\eta_k}(t)$ for all $t$. First remark that
	\begin{align*}
	\nu\left(G_{\eta_{k+1}}^C(t)\right)&=\nu\left(\left\{s\in T:\delta X(s,\alpha(t))<\eta_{k+1}\right\}\right)\\
	&=\nu\left(\left\{s\in T:\delta X(s,\alpha(t))<\frac{0+\eta_k}{2}\right\}\right)\\
	&=\nu\left(\left\{s\in T:\delta X(s,\alpha(t))<\frac{x_1+x_2}{2}\right\}\right).\\
        \end{align*}
        Applying now condition A, we obtain
        \begin{align*}
	\nu\left(G_{\eta_{k+1}}^C(t)\right)&\leq \Lambda \left(\nu\left(\left\{s\in T:\delta X(s,\alpha(t))<x_1\right\}\right)^\phi+\nu\left(\left\{s\in T:\delta X(s,\alpha(t))<x_2\right\}\right)^2\right)\\
	&\leq \Lambda \left(\nu\left(\left\{s\in T:\delta X(s,\alpha(t))<0\right\}\right)^\phi+\nu\left(\left\{s\in T:\delta X(s,\alpha(t))<\eta_k\right\}\right)^2\right)\\
	&\leq \Lambda \left(\nu\left(G_0^C(t)\right)^\phi +\nu\left(G_{\eta_k}^C(t) \right)^2\right)\\
	\end{align*}
	From \eqref{eq:loss_and_accuracy},
	\begin{equation*}
	\nu(G_0^C(t))=1-\acc(t)\leq\frac{\bar L(t_0)}{\log 2}.
	\end{equation*}
	Therefore,
	\begin{equation}\label{eq:recursive_good_set_size}
	\nu\left(G_{\eta_{k+1}}^C(t)\right)\leq \Lambda \left(\left(\frac{\bar L(t_0)}{\log 2}\right)^\phi+\nu\left(G_{\eta_k}^C(t)\right)^2\right)
	\end{equation}
	Since $\eta_k$ is decreasing with $k$, $\nu(G_{\eta_k}^C(t))$ is also decreasing in $k$. This means that if for some $k_0$,
	\begin{equation}\label{eq:k0}
	\nu\left(G_{\eta_{k_0}}^C(t)\right)\leq \sqrt{2}\left(\frac{\bar L(t_0)}{\log 2}\right)^{\phi/2},
	\end{equation}
	then \eqref{eq:k0} also holds with $\eta_{k_0}$ replaced by $\eta_k$ for all $k>k_0$. Therefore, for all $k>k_0$, we may use \eqref{eq:recursive_good_set_size} to estimate:
	\begin{equation}\label{eq:good_set_bound_for_small_eta}
	\nu\left(G_{\eta_k}^C(t)\right)\leq \Lambda \left(\left(\frac{\bar L(t_0)}{\log 2}\right)^\phi+\nu\left(G_{\eta_{k-1}}^C(t)\right)^2\right)\leq \Lambda \left(\left(\frac{\bar L(t_0)}{\log 2}\right)^\phi+2\left(\frac{\bar L(t_0)}{\log 2}\right)^\phi\right)=3\Lambda \left(\frac{\bar L(t_0)}{\log 2}\right)^\phi.
	\end{equation}
	On the other hand, if $k\leq k_0$,
	\begin{equation*}
	\nu\left(G_{\eta_{k}}^C(t)\right)>\sqrt{2}\left(\frac{\bar L(t_0)}{\log 2}\right)^{\phi/2},
	\end{equation*}
	or equivalently,
	\begin{equation*}
	\left(\frac{\bar L(t_0)}{\log 2}\right)^\phi<\frac{1}{2}\nu\left(G_{\eta_{k}}^C(t)\right)^2.
	\end{equation*}
	Thus, for $k<k_0$, we can use \eqref{eq:recursive_good_set_size} to obtain
	\begin{equation*}
	\nu(G_{\eta_k}^C(t))\leq \Lambda \left(\left(\frac{\bar L(t_0)}{\log 2}\right)^\phi+\nu\left(G_{\eta_{k-1}}^C(t)\right)^2\right)\leq \Lambda \left(\frac{1}{2}\nu\left(G_{\eta_{k-1}}^C(t)\right)^2+\nu\left(G_{\eta_{k-1}}^C(t)\right)^2\right)=\frac{3}{2}\nu\left(G_{\eta_{k-1}}^C(t)\right)^2.
	\end{equation*}
	By induction,
	\begin{equation*}
	\nu(G_{\eta_k}^C(t))\leq\left(\frac{3}{2}\Lambda \right)^{2^k-1}\nu\left(G_{\eta_0}(t)\right)^{2^k}.
	\end{equation*}
	From \eqref{eq:principal_good_set}, $\nu\left(G_{\eta_0}(t)\right)\leq 1/2\Lambda $, so
	\begin{equation}\label{eq:good_set_bound_for_large_eta}
	\nu\left(G_{\eta_k}^C(t)\right)\leq\left(\frac{3}{2}\Lambda \right)^{2^k-1}\left(\frac{1}{2\Lambda }\right)^{2^k}\leq\left(\frac{3}{4}\right)^{2^k}=\left(\frac{3}{4}\right)^{-\frac{\log(3\Lambda \bar L(t_0))}{\eta_k}}.
	\end{equation}
	For any $k$, either \eqref{eq:good_set_bound_for_small_eta} and \eqref{eq:good_set_bound_for_large_eta} holds, so for all $k\geq 0$,
	\begin{equation}\label{eq:discrete_good_set_bound}
	\nu\left(G_{\eta_k}^C(t)\right)\leq\left(\frac{3}{4}\right)^{-\frac{\log(3\Lambda \bar L(t_0))}{\eta_k}}+3\Lambda \left(\frac{\bar L(t_0)}{\log 2}\right)^\phi.
	\end{equation}
	We wish to find a bound on good sets for all $\eta^\ast$ with $0<\eta^\ast<-\log(3\Lambda \bar L(t_0))$, not just $\eta^\ast=\eta_k$ for some $k$. By the monotonicity of good sets, if $\eta_{k+1}<\eta^\ast\leq\eta_k$,
	\begin{equation}\label{eq:good_set_monotinicity}
	\nu\left(G_{\eta_{k+1}}^C(t)\right)\leq\nu\left(G_{\eta^\ast}^C(t)\right)\leq\nu\left(G_{\eta_{k}}^C(t)\right).
	\end{equation}
	Since $\eta_{k}=2\eta_{k+1}$, $1/2\eta^\ast<1/\eta_k$. Applying \eqref{eq:discrete_good_set_bound} to \eqref{eq:good_set_monotinicity}, we have
	\begin{equation*}
	\nu\left(G_{\eta^\ast}^C(t)\right)\leq\left(\frac{3}{4}\right)^{-\frac{\log(3\Lambda \bar L(t_0))}{\eta_k}}+3\Lambda \left(\frac{\bar L(t_0)}{\log 2}\right)^\phi\leq\left(\frac{3}{4}\right)^{-\frac{\log(3\Lambda \bar L(t_0))}{2\eta^\ast}}+3\Lambda \left(\frac{\bar L(t_0)}{\log 2}\right)^\phi.
	\end{equation*}
	Since $\nu(G_{\eta^\ast}(t))=1-\nu(G_{\eta^\ast}(t))$, for all $\eta^\ast$ with $0<\eta^\ast<\log(3\Lambda \bar L(t_0))$,
	\begin{equation}\label{eq:direct_good_set_bound}
	\nu\left(G_{\eta^\ast}(t)\right)>1-\left(\frac{3}{4}\right)^{-\frac{\log(3\Lambda \bar L(t_0))}{2\eta^\ast}}-3\Lambda \left(\frac{\bar L(t_0)}{\log 2}\right)^\phi,
	\end{equation}
	and since $G_0(t)=\cup_{\eta^\ast>0}G_{\eta^\ast}(t)$,
	\begin{equation}\label{eq:direct_accuracy_bound}
	\begin{split}
	\acc(t)&=\nu\left(G_0(t)\right)\\
	&=\lim_{\eta^\ast\to 0}\nu\left(G_{\eta^\ast}(t)\right)\\
	&>\lim_{\eta^\ast\to\infty}\left(1-\left(\frac{3}{4}\right)^{-\frac{\log(3\Lambda \bar L(t_0))}{2\eta^\ast}}-3\Lambda \left(\frac{\bar L(t_0)}{\log 2}\right)^\phi\right)\\
	&=1-3\Lambda \left(\frac{\bar L(t_0)}{\log 2}\right)^\phi.
	\end{split}
	\end{equation}
	Finally, by choosing $\eta_0$ sufficiently large and $\eta>\eta_0$, we can make $\bar L(t_0)$ sufficiently large that $3\Lambda (\bar L(t_0)/\log 2)^\phi<\varepsilon$. Therefore,
	\begin{equation*}
	\acc(t)>1-\varepsilon
	\end{equation*}
	and
	\begin{equation*}
	\nu\left(G_{\eta^\ast}(t)\right)>1-\epsilon-\left(\frac{3}{4}\right)^{-\frac{\log(3\Lambda \bar L(t_0))}{2\eta^\ast}},
	\end{equation*}
        concluding the proof.
%
\subsection{Upper bound on the loss for condition B}
%
Just as for Theorem \ref{thm:stabilityA}, the first step in the proof of Theorem \ref{thm:stabilityB} is to derive an upper bound on the loss function based now on condition B.
\begin{lemma}\label{lem:loss_bound_2}
	Suppose that $T\subset\mathbb R^n$ with weights $\nu(s)$ is a training set for a softmax DNN which satisfies condition B in the sense of definition \ref{def:cond_B}, at time $t_0$ for $m_0$ and some constants $\kappa$, $\delta>0$ and $\sigma>0$. For $I_0$ corresponding to the point $x=0$, if for some $\eta<I_0$, $\delta_0>0$, and $\delta_0<\delta$,
	\begin{equation}\label{eq:good_set_big_2a}
	\nu(G_\eta(t_0))>1-\delta_0
	\end{equation}
	and
	\begin{equation}\label{eq:bad_set_empty_2a}
	B_{-1}(t_0)=\emptyset,
	\end{equation}
	then the cross-entropy loss is bounded by: 
	\begin{equation}\label{eq:loss_bounded_2}
\bar L(t_0)\leq C\,m_0+(K-1)\left(e^{-\eta}+\delta_0\,e^{-\eta/\kappa} +\delta_0\,\frac{\sigma+1}{\eta^\gamma}\,\left(e+2\right)\right),
	\end{equation}
	for any $0<\gamma\leq\min\{1,\log(1+\sigma)/\log\kappa\}$.
\end{lemma}
\begin{proof}
Without loss of generality, we may assume that $\eta\geq 1$.  Consider $\ell>0$ such that $I:=[0,\ell)\subset [-1,\eta)$. Since $\eta<I_0$, we may apply condition B at point $0$ and we have either that
        \[
\nu\left(\left\{s\in T:\delta X(s,\alpha(t_0))\in I\right\}\right)\leq m_0,
\]
or that
	\begin{equation*}
	\nu\left(\left\{s\in  T:\delta X(s,\alpha(t_0))\in \kappa\,I\right\}\right)\geq\min\{\delta,\;(1+\sigma)\,\nu\left(\left\{s\in T:\delta X(s,\alpha(t_0))\in I\right\}\right)\}.
	\end{equation*}
	Applying condition B repeatedly $j$ times in this last case, we conclude that
	\begin{equation*}
	\nu\left(\left\{s\in  T:\delta X(s,\alpha(t_0))\in \kappa^j\,I\right\}\right)\geq\min\{\delta,\;(1+\sigma)^j\nu\left(\left\{s\in  T:\delta X(s,\alpha(t_0))\in I\right\}\right)\}.
	\end{equation*}
	By \eqref{eq:good_set_big_2a}, $\nu\left(\left\{s\in T:\delta X(s,\alpha(t_0))<\eta\right\}\right)\leq\delta_0$. Therefore, provided $\kappa^j I\subset (-\infty,\;\eta)$,
	\begin{equation*}
	\delta_0\geq\nu\left(\left\{s\in T:\delta X(s,\alpha(t_0))\in \kappa^j\, I\right\}\right)\geq\min\{\delta,\;(1+\sigma)^j\nu\left(\left\{s\in  T:\delta X(s,\alpha(t_0))\in I\right\}\right)\}.
	\end{equation*}
Since $\delta_0<\delta$, we conclude that
	\begin{equation}\label{eq:interval_size_estimate_1}
	\nu\left(\left\{s\in  T:\delta X(s,\alpha(t_0))\in I\right\}\right)\leq \max\left(m_0,\;\frac{\delta_0}{(\sigma+1)^j}\right)
	\end{equation}
	for all $j$ so that $\kappa^j I\subset(-\infty,\eta)$, or equivalently, for all $j$ so that $\kappa^j\,\ell\leq\eta$. Obviously, \eqref{eq:interval_size_estimate_1} is best for $j$ as large as possible with the largest value given by
	\begin{equation*}
j_{\text{max}}=\left\lfloor\frac{\log\left(\frac{\eta}{\ell}\right)}{\log \kappa}\right\rfloor>\frac{\log\left(\frac{\eta}{\ell}\right)}{\log\kappa}-1.
	\end{equation*}
	Therefore,
	\begin{equation*}
	\begin{split}
	\nu\left(\left\{s\in  T:\delta X(s,\alpha(t_0))\in I\right\}\right)&\leq m_0+\frac{\delta_0}{(\sigma+1)^{j_{\text{max}}}}\\
	&<m_0+\frac{\delta_0}{(\sigma+1)^{\frac{\log\left(\frac{\eta}{\ell}\right)}{\log \kappa}-1}}\\
        \end{split}
	\end{equation*}
        We have thus proved the bound
        \begin{equation}\label{eq:interval_size_estimate_2}
\nu\left(\left\{s\in  T:\delta X(s,\alpha(t_0))\in I\right\}\right)	<m_0+\delta_0\,(\sigma+1)\left(\frac{\ell}{\eta}\right)^{\gamma},\\
\end{equation}	
	where $\gamma=\log(\sigma+1)/\log \kappa$. If condition B holds for some $\sigma$, it also holds for all smaller $\sigma$, so we may assume without loss of generality that $\sigma<\kappa-1$, and therefore $0<\gamma<1$. Similarly we may assume that $\kappa\geq 2$.
	
	Now we will apply \eqref{eq:interval_size_estimate_2} to explicit intervals. Let $p=\lfloor\log(\eta)/\log\kappa\rfloor$ and let $I_i=[\kappa^i,\kappa^{i+1})$ for all $i\in \N$ with $i<p$. We also define $I_{-1}=[-1,1)$ and $I_p=[\kappa^p,\ \eta)$. For $i\geq 0$, $I_i\in [0,\ \kappa^{i+1}]$ with $\ell_i=\kappa^{i+1}$. Therefore by \eqref{eq:interval_size_estimate_2},
	\begin{equation}\label{eq:interval_size_estimate_3}
	\nu\left(\left\{s\in T:\delta X(s,\alpha(t_0))\in I_i\right\}\right)\leq m_0 +\delta_0\,(1+\sigma)\left(\frac{\ell_i}{\eta}\right)^{\gamma}.
	\end{equation}
	The interval $I_{-1}$ is centered at $0$ and has width $1$, so
	\begin{equation}\label{eq:interval_size_estimate_4}
	\nu\left(\left\{s\in T:\delta X(s,\alpha(t_0))\in I_{-1}\right\}\right)\leq m_0 +\delta_0\,(1+\sigma)\left(\frac{1}{\eta}\right)^{\gamma}.
	\end{equation}
        Finally since $\kappa^p>\kappa^{\log\eta/\log\kappa-1}$, we simply bound for $I_p$ 
        \begin{equation}\label{eq:interval_size_estimate_5}
	\nu\left(\left\{s\in T:\delta X(s,\alpha(t_0))\in I_{p}\right\}\right)\leq \delta_0.
	\end{equation}
	For each $s\in T$, either $\delta X(s,\alpha(t_0))\geq \eta$, or $\delta X(s,\alpha(t_0))\in I_i$ for integer $i\geq-1$. Therefore,
	\begin{equation*}
	\begin{split}
	\bar L(t_0)&\leq\sum_{s\in T}\nu(s)\log\left(1+(K-1)e^{-\delta X(s,\alpha(t_0))}\right)\\
	&=\sum_{\substack{s\in T\\ \delta X(s,\alpha(t_0))\geq\eta}}\nu(s)\log\left(1+(K-1)e^{-\delta X(s,\alpha(t_0))}\right)+\sum_{i=-1}^{p-1}\sum_{\substack{s\in T\\ \delta X(s,\alpha(t_0))\in I_i}}\nu(s)\log\left(1+(K-1)e^{-\delta X(s,\alpha(t_0))}\right)\\
	&\leq\log(1+(K-1)e^{-\eta})\sum_{\substack{s\in T\\ \delta X(s,\alpha(t_0))\geq\eta}}\nu(s)+\sum_{i=-1}^{p-1}\log\left(1+(K-1)e^{-\inf I_i}\right)\sum_{\substack{s\in T\\ \delta X(s,\alpha(t_0))\in I_i}}\nu(s).\\
        \end{split}
        \end{equation*}
Of course by decomposing       
	\begin{equation*}
	\begin{split}
	\bar L(t_0)&	\leq(K-1)e^{-\eta}\nu(G_\eta(t_0))+\sum_{i=-1}^{p-1}\log\left(1+(K-1)e^{-\inf I_i}\right)\nu\left(\left\{s\in T:\delta X(s,\alpha(t_0))\in I_i\right\}\right)\\
	&\leq(K-1)\,e^{-\eta}+\log\left(1+(K-1)e\right)\nu\left(\left\{s\in T:\delta X(s,\alpha(t_0))\in I_{-1}\right\}\right)\\
        &+(K-1)\,e^{-\eta/\kappa}\,\nu\left(\left\{s\in T:\delta X(s,\alpha(t_0))\in I_{p}\right\}\right)\\
        &+(K-1)\sum_{i=0}^{p-1}e^{-\kappa^i}\nu\left(\left\{s\in T:\delta X(s,\alpha(t_0))\in I_i\right\}\right).\\
\end{split}
	\end{equation*}
We now use \eqref{eq:interval_size_estimate_3}, \eqref{eq:interval_size_estimate_4} and \eqref{eq:interval_size_estimate_5} to derive  
        	\begin{equation}\label{eq:cond_B_loss_estimate}
	\begin{split}
	  \bar L(t_0)
          \leq & C\,m_0+(K-1)e^{-\eta}+(K-1)\,\delta_0\,e^{-\eta/\kappa}+\delta_0\,(\sigma+1)\,\frac{\log(1+e\,(K-1))}{\eta^\gamma}\\
&       +\delta_0\,(\sigma+1)(K-1)\sum_{i=0}^{p-1} \left(\frac{\kappa^{i}}{\eta}\right)^\gamma\, e^{-\kappa^i}.
	\end{split}
	\end{equation}
	We need to estimate the value of the sum above. 
        By Lemma \ref{lem:double_exponential_sum_constant_2},
we have that, since $\gamma<1$, 	
        \begin{equation}\label{eq:cond_B_loss_sum_estimate2}
	\begin{split}
          \sum_{0=1}^{p-1} \left(\frac{\kappa^{i}}{\eta} \right)^\gamma\, e^{-\kappa^i}&= \frac{1}{\eta^\gamma}\,\sum_{i=0}^p \kappa^{i\,\gamma}\,e^{-\kappa^i}\leq \frac{1}{\eta^\gamma}\,\sum_{i=0}^p \kappa^{i}\,e^{-\kappa^i}\leq  \frac{2}{\eta^\gamma}.
        \end{split}
	\end{equation}
	Applying \eqref{eq:cond_B_loss_sum_estimate2} to \eqref{eq:cond_B_loss_estimate}, we arrive at
\[
\bar L(t_0)\leq C\,m_0+(K-1)\left(e^{-\eta}+\delta_0\,e^{-\eta/\kappa} +\delta_0\,\frac{\sigma+1}{\eta^\gamma}\,\left(e+2\right)\right),
\]
which finishes the proof.
\end{proof}
\subsection{Proof of Theorem \ref{thm:stabilityB}}
We start with the trivial bound derived from \eqref{eq:loss_estimates} by a sort of Chebyshev inequality
\begin{equation}
\sum_{s\in G_{\eta^*}^c(t)} \nu(s)\,\log (1+e^{-\eta^*})\leq \bar L(t)\leq \bar L(t_0).
\label{chebyshev}
\end{equation}
As a consequence, we obtain from Lemma \ref{lem:loss_bound_2} that provided $G_{-1}^c(t_0)=\emptyset$ and $\nu(G_\eta(t_0))>1-\delta_0$,
\[
\nu(G_0^c(t))\leq C\,m_0+\frac{K-1}{\log 2}\,\left(e^{-\eta}+\delta_0\,e^{-\eta/\kappa} +\delta_0\,\frac{\sigma+1}{\eta^\gamma}\,\left(e+2\right)\right).
\]
{We can make sure that the right-hand side is less than $\eps$ if for some constant $C(K,\sigma,\kappa)$
\[
m_0\leq \frac{\eps}{2C},\quad \eta\geq \kappa\log \frac{1}{\eps}+C,\quad \delta_0\leq \eps\frac{\eta^\gamma}{C}.
\]
Of course one could even be rather explicit on $C$
\[
C=\max\left(\kappa\log \frac{6(K-1)}{\log 2}, \;6\frac{K-1}{\log 2}\,(\sigma+1)\,(e+2)\right).
\]
This immediately proves the first part of Theorem \ref{thm:stabilityB}.}

For the second part, we note that the above choice of $\eta$ and $\delta_0$ also guarantees that
\[
\bar L(t_0)\leq \eps\,\log 2.
\]
Therefore by \eqref{chebyshev} and for $\eta^*\geq 0$, we have that
\[
\nu(G_{\eta^*}^C(t))\leq \eps\,\frac{\log 2}{\log (1+e^{-\eta^*})}\leq 2\,\log 2\,\eps\,e^{\eta^*}.
\]

\medskip

We finish the proof by a technical remark which may, in some cases, improve the estimates. Define a strip
\[
S=\{s\,:\;0\leq \delta X(s,\alpha(t))\leq \ell\}.
\]
We may directly apply the bound \eqref{eq:interval_size_estimate_2}, proved previously, which we recall below: For any $\eta>\ell$,
\[
\nu(S)<m_0+\nu(G_{\eta}^c(t))\,(\sigma+1)\,\frac{\ell^\gamma}{\eta^\gamma}.
\]
This implies that
\begin{equation*}
\nu(S)<m_0+2\,\log 2\,\eps\,e^{\eta}\,(\sigma+1)\,\frac{\ell^\gamma}{\eta^\gamma}.
\end{equation*}
One may optimize in $\eta$ by finding the minimum of
\[
f(\eta)=\frac{e^\eta}{\eta^\gamma},
\]
which is obtained at $\eta=\gamma$. Therefore 
\begin{equation}
  \nu(S)<\left\{\begin{aligned} &m_0+2\,\log 2\,\eps\,\ell^\gamma\,(\sigma+1)\,\frac{e^\gamma}{\gamma^\gamma}\quad \mbox{if}\ \ell\leq\gamma,\\
& m_0+2\,\log 2\,\eps\,e^{\ell}\,(\sigma+1)\quad \mbox{if}\ \ell>\gamma.
  \end{aligned}\right.\label{improvedbound}
\end{equation}
This of course has to be compared with the trivial bound
\[
\nu(S)\leq \nu(G_{\ell}^c)\leq 2\,\log 2\,\eps\,e^{\ell},
\]
which makes it obvious that \eqref{improvedbound} is only useful if $\ell$ is small enough.
\subsection{Proof of Theorem \ref{thm:data_clusters_to_confidences}}
The proof is performed by induction on the number of layers in the network. This is the reason why we consider the no small data clusters condition. We recall here the definition of truncated ellipsoidal cylinders
\begin{equation}
E=E(Q,s,v_1,t_1,\ldots,v_k,t_k)=\left\{x,\ (x-s)^T\,Q\,(x-s)\leq 1\ \mbox{and}\ v_i\cdot x\leq t_i\ \forall i=1,\ldots, k\right\},\label{truncatedcylinder}
\end{equation}
for some symmetric {positive semi-definite} matrix $Q$. The dilated cylinders by some factor $\kappa$ are obtained with
        \begin{equation}
\kappa\,E=\left\{x,\ (x-s)^T\,Q\,(x-s)\leq \kappa^2\ \mbox{and}\ v_i\cdot x\leq t_i\ \forall i=1,\ldots, k\right\}.\label{dilatedcylinder}
        \end{equation}
        We then recall the no small data clusters condition on any measure $\mu$ for given $r$, $k$, $m_0$, $\kappa$, $\sigma$ and $\delta$, 
        \begin{equation}
          \begin{split}
            &\forall Q\in M_d(\R^d)\ \mbox{with}\ \mbox{rank}\,Q\leq r,\; \forall v_1,\ldots,v_k\in \R^d\setminus\{0\},\;\forall s\in \R^d,\;\forall t_1,\ldots,t_k\in \R,\ \exists \eps_0\ \mbox{s.t.}\ \forall \eps\leq \eps_0,\\
            &\qquad \mu\left(\kappa\,\eps E(u,s,v_1,t_1,\ldots,v_k,t_k)\right)\geq\min\Big\{\delta,\; \max\big\{m_0,\ (1+\sigma)\,\mu\left(\eps\,E(u,s,v_1,t_1,\ldots,v_k,t_k)\right)\big\}-m_0\Big\}.
\end{split}
\label{doublingstrip}
  \end{equation}
        A first key point is to observe that reasonable distributions $\mu$ satisfy such assumptions. The following theorem shows that NSDC condition \eqref{doublingstrip} is satisfied in a simple case when data is concentrated on smooth manifolds that are non-degenerate in some appropriate sense related to classical transversality conditions with curvature requirements.

More precisely, we will call a smooth manifold $M$ of dimension $d_0$ non degenerate with no flat parts of dimension $r$ in the present context if:
        \begin{itemize}
          \item for any affine plane $P$ of dimension $r$ or more, the intersection $M\cap P$ is composed of a finite number of connected components, (the intersection of $P$ and $M$ consists of ``finitely many pieces'').
          \item The neighborhood of M around $P$ can be parametrized smoothly around each component of $M\cap P$. Specifically, for some $\bar \eps$, there exists a finite number of balls $B_i$, corresponding smooth applications $f_i:\;[0,\ 1]^{d_0}\to B_i$ and probability measures $\nu_i$ s.t. 
            \begin{equation}
              \begin{split}
                &- M\cap \{x\,|\;d(x,P)\leq \eps_0\} =\bigcup_i f_i([0,\ 1]^{d_0}),\\
                &- M\cap P=\bigcup_i f_i(\{0\}\times [0,\ 1]^{d_0-1}),\\
                &- \exists C,\ \forall \;\mbox{Borel set}\ O\subset B_i,
                \quad  \mathcal{H}^{d_0}|_M(O) = \int_{[0,\ 1]^{d_0-1}} |\{t\,|\; f_i(t,y)\in O\}|\,\nu_i(dy),\\
                &- \forall (t,y)\in [0,\ 1]^{d_0},\quad (\partial_t f(t,y),\ldots, \partial^d_t f(t,y))\ \mbox{spans}\ \R^d.
                \end{split}\label{nondegeneratecurvature}
              \end{equation}
        \end{itemize}        
        {We next provide a heuristic explanation of the above two conditions. To explain the first bullet point recall that we interpreted $\epsilon_0$ in \eqref{eq:toy_cond} as the scale at which non-flatness is resolved by a doubling condition. Roughly speaking, on each component $M_i$ of $M\cap P$ the non-flatness is verified by its own resolution $\epsilon_0^i$, and the non-flatness of all of $M\cap P$ is verified by the infimum of $\epsilon_0^i$ over all $i$. If the number of components $M^i$ is infinite, this infimum could be zero, so we require that $M\cap P$ has at most finitely many components. The second bullet point essentially means that if $M\subset \mathbb R^d$ has dimension  $d_1$, and $P\subset \mathbb R^d$ is a hyperplane of dimension $d_2$ with $d_1\leq d_2<d$, then $\text{dim}\,M\cap P<d_1$. In this sense, no hyperplane contains a substantial portion  of $M$ (a flat part of $M$).
        For example, it appears that if the manifold $M$ is of co-dimension $1$ in $\mathbb R^d$, then \eqref{nondegeneratecurvature} implies that Gaussian curvature defined by the second derivative does not become zero, which intuitively means no flat parts. Of course typical datasets may be concentrated on low dimensional manifolds, with hence larger co-dimension. 
         However, if $\text{codim}\, M>1$, nonzero curvature is not sufficient. For example, if $M$ is a circle in $\mathbb R^3$ lying  in a 2-dimensional plane, then $M$ is flat in the sense of \eqref{nondegeneratecurvature} even thought the curvature is non-zero. 
         While curvature conditions  are often  used in literature (e.g., \cite{Chen}) the curvature is a property of the manifold  while our  condition \eqref{dilatedcylinder} concerns both the manifold and the ambient space.}
     
       
        We emphasize in particular the importance of the last condition in \eqref{nondegeneratecurvature} which controls more than just the curvature of $M$ and guaranteed that it will extend in every direction.

        Assumption \eqref{nondegeneratecurvature} could be relaxed somewhat as it would not be a problem for example if the whole manifold was contained in some lower dimensional plane: Issues only occur if an arbitrarily small part of $M$ is flat while the whole manifold is not. Similarly the last condition in \eqref{nondegeneratecurvature} could be relaxed by asking that $(\partial_t f(t,y),\ldots, \partial^k_t f(t,y))$ spans $\R^d$ for $k$ large enough. As it is already technical, we preferred to keep the simpler Assumption \eqref{nondegeneratecurvature} in the context of this article.

        Finally we point out that the $f_i$ cannot in general be one to one, unless $M\cap P$ is of dimension exactly $d_0-1$. Otherwise we will always have pairs of points $y,\;y'\in [0,\ 1]^{d_0}$ s.t. $f_i(0,y)=f_i(0,y')$. The $f_i$ simply offer a convenient way of parametrizing $M$ in the neighborhood of $P$ by a set of 1-dimensional curves.  
        
        \begin{theorem}
Assume that the measure $\mu$ satisfied $\mu=C\,\mathcal{H}^{d_0}|_M$ where $M$ is non-degenerate with no flat parts in the sense of \eqref{nondegeneratecurvature}. Then $\mu$ satisfies \eqref{doublingstrip} for the corresponding $r$, $m_0=0$, $\delta=1$ and any $k$, any $\kappa$ and any $\sigma<\kappa^{1/d}$.
          \end{theorem}
        \begin{proof}
         Choose any basic truncated ellipsoidal cylinders, {\em i.e.} parameters $Q,v_1,\ldots,v_k,s,t_1,\ldots, t_k$. Without loss of generality, we may assume that $s=0$. Moreover since the problem is invariant by rotation, we may also assume that $Q$ can be diagonalized according to the basis,
\[
x^T\,Q\,x=\lambda_1\,x_1^2+\ldots+\lambda_l\,x_l^2,
\]
with in particular $\lambda_{l+1}=\ldots=\lambda_d=0$ and $\lambda_1,\ldots,\lambda_l>0$ {since $Q$ is positive semi-definite}. This leaves the possibility that the ellipsoid extends over the whole domain in the $d-l$ last directions. We denote by $P$ the plane of equation $x_1=\ldots=x_l=0$ of dimension $d-l$. Since $l\leq r$, $d-l\geq d-r$ and \eqref{nondegeneratecurvature} applies to $M$ and $P$.

%
First notice that if $v_{i,l+1},\ldots, v_{i,d+1}=0$ and $t_i\neq 0$, we observe that if $x\in \kappa\,\eps\,E$ then
\[
\sum_{i\leq l} x_i^2\leq \frac{\kappa^2\,\eps^2}{\lambda},\quad \lambda=\inf_{i\leq l} \lambda_i.
\]
Therefore as $\eps\to 0$ then $v_i\cdot x\to 0$ and that implies that the constraint $v_i\cdot x\leq t_i$ is either always satisfied if $t_i>0$ or never satisfied if $t_i<0$. Consequently we may freely limit ourselves to consider the constraints with $t_i=0$ in that case.

Moreover this implies that $E_{\kappa\,\eps}\subset \{x\,|\;d(x,P)\leq \kappa\,\eps/\sqrt{\lambda}\}$. This leads us to only consider $\eps\leq \sqrt{\lambda}\,\bar\eps$ so that \eqref{nondegeneratecurvature} provides an adequate parametrization of $M\cap E_{\kappa\,\eps}$ around a finite number of components which we can consider independently.

Denote by $\mathcal{C}$ the subset of $x\in \R^d$ s.t. $x\cdot v_i\leq t_i$ {for all $i=1,\cdots,k$}. Further denote for simplicity $E_\eps=\{x\,|\;x^T\,Q\,x\leq \eps^2\}$. Because of the third point in \eqref{nondegeneratecurvature} is to estimate for every $y$ the quantity $|\{t\,|\;f_i(t,y)\in E_\eps\cap \mathcal{C}\}$ in terms of $\eps$.

Fix any $y$ and use the last point in condition~\eqref{nondegeneratecurvature} and denote by $k_0=k_0(y)$ the first index s.t. $\partial^{k_0}_tf_{i,j}(0,y)\neq 0$ for some $j\leq l$ where $f_{i,j}$ is the $j$-th coordinate of function $f_i$. We always have $k_0\leq d$ because of \eqref{nondegeneratecurvature}.

Observe that since $f_i(0,y)\in P$ by the second point in \eqref{nondegeneratecurvature}, with this index $k_0$, one has that for some constant $C$ independent of $y$ and $t$
\begin{equation}\label{eq:taylor_expansion_of_M}
|f_i(t,y)^T\, Q\,f_i(t,y)-\gamma\,t^{2k_0}|\leq C\,t^{2k_0+1},\quad \gamma=\frac{1}{(k_0!)^2}\,(f_i^{k_0}(0,y))^T\,Q\,f_i^{k_0}(0,y)>\lambda\,\sum_{i=1}^l |f^{k_0}_i(0,y)|^2.
\end{equation}
From \eqref{eq:taylor_expansion_of_M}, we observe that $\gamma t^{2k_0}-Ct^{2k_0+1}\leq f_i(t,y)^T Q f_i(t,y)$. Let
\begin{equation}
a(\varepsilon)=\left(\frac{2\varepsilon^2}{\gamma}\right)^{\frac{1}{2k_0}}.
\end{equation}
If $\varepsilon$ is sufficiently small, then $a(\varepsilon)<\gamma/2C$. If $a(\varepsilon)<t<\gamma/2C$, then
\begin{align*}
f_i(t,y)^T Q f_i(t,y)&\geq\gamma t^{2k_0}-Ct^{2k_0+1}=t^{2k_0}\,(\gamma-C\,t)\geq \frac{\gamma}{2} t^{2k_0}\\
&\geq \frac{\gamma}{2} (a(\eps))^{2k_0}=\varepsilon^2.
\end{align*}


This first shows that if $t\leq \gamma/2C$  and $t\geq a(\eps)=(2\,\eps^2/\gamma)^{1/2k_0}$ then $f_i(t,y)$ cannot belong to $E_\eps$. The same argument can be applied to any point $(t_0,y)$ instead of $(0,y)$ of course. This proves in particular that $f_i(t,y)$ only intersects $P$ at a finite number of points and that $\{t\,|\;f_i(t,y)\in E_\eps\}$  is composed of a finite union of intervals centered around each of those points. From the second point in condition \eqref{nondegeneratecurvature}, each of points $(t_0,y)$ s.t. $f_i(t,y)\in P$ is of the form $(0,y')$ for some other $y'$. As a consequence we may freely assume that $t\leq \gamma/2C$ and therefore that $t\geq (2\,\eps^2/\gamma)^{1/2k_0}$ if $f_i(t,y)\in E_\eps$.

We can then easily obtain a precise characterization of $\{t\,|\;f_i(t,y)\in E_\eps\}$ from the asymptotic expansion above with
\begin{equation}
  \begin{split}
    &f_i(t,y)\in \{t\,|\;f_i(t,y)\in E_\eps\}\implies |t|\leq (2\eps^2/\gamma)^{1/2k_0}+C\,\eps^{2/k_0},\\
& |t|\leq (2\eps/\gamma)^{1/2k_0}-C\,\eps^{2/k_0}\implies f_i(t,y)\in \{t\,|\;f_i(t,y)\in E_\eps\}. 
  \end{split}
    \label{measureEeps}
\end{equation}
We now come back to any condition $f(t,y)\cdot v_j\leq t_j$, we see that if $f_i(0,y)\cdot v_j\neq t_j$ then this condition is either automatically satisfied or never satisfied for $\eps$ small enough. Denote by $\bar\eps(y)$ the corresponding cap on $\eps$. Consequently, if $\eps\leq \bar\eps(y)$, we may freely assume that $f_i(0,y)\cdot v_j=t_j$ for all $j$. 

A similar
consequence is that we can characterize easily the part of set $\mathcal{C}$ that is within $E_{\eps}$ for any $\eps$. Given any $v_j$, we denote by $k_j=k_j(y)$ the first index s.t. $\partial_t^{k_j}f_i(0,y)\cdot v_j\neq 0$, which again exists thanks to the last point in \eqref{nondegeneratecurvature}. Then we have as before that
\[
\left|v_j\cdot f_i(t,y)-t_j-\frac{\partial_t^{k_j}f_i(0,y)\cdot v_j}{k_j!}\,y^{k_j}\right|\leq C\,y^{k_j+1},
\]
where $C$ is again independent of $t$ and $y$.

This implies that for $\eps\leq \bar\eps(y)$ small enough, $\{t\,|\;f_i(t,y)\cdot v_j\leq t_j\}\cap [-(2\eps^2/\gamma)^{1/2k_0},\ (2\eps^2/\gamma)^{1/2k_0}]$ is exactly an interval $I_{j,\eps}=I_{j,\eps}(y)$ with
\begin{itemize}
  \item $I_{j,\eps}(y)=[0,\ (2\eps^2/\gamma)^{1/2k_0}]$ if $k_j(y)$ is odd and $\partial_t^{k_j}f_i(0,y)\cdot v_j<0$;
  \item $I_{j,\eps}=[-(2\eps^2/\gamma)^{1/2k_0},\ 0]$ if $k_j(y)$ is odd and $\partial_t^{k_j}f_i(0,y)\cdot v_j>0$;
  \item $I_{j,\eps}=[-(2\eps^2/\gamma)^{1/2k_0},\ (2\eps^2/\gamma)^{1/2k_0}]$ if $k_j(y)$ is even and $\partial_t^{k_j}f_i(0,y)\cdot v_j<0$;
  \item $I_{j,\eps}=\{0\}$ if $k_j(y)$ is even and $\partial_t^{k_j}f_i(0,y)\cdot v_j<0$.
\end{itemize}
Denoting by $I_\eps(y)$ the intersection of all $I_{j,\eps}$ we have that $\{t\,|\;f_i(t,y)\in \mathcal{C}\}\cap [-(2\eps^2/\gamma)^{1/2k_0},\ (2\eps^2/\gamma)^{1/2k_0}]$ is precisely the interval $I_\eps$. Moreover $I_\eps$ is either $\{0\}$ (in which case we have nothing to prove) or one of $[0,\ (2\eps^2/\gamma)^{1/2k_0}]$, $[-(2\eps^2/\gamma)^{1/2k_0},\ 0]$, $[-(2\eps^2/\gamma)^{1/2k_0},\ (2\eps^2/\gamma)^{1/2k_0}]$. And we finally emphasize that the choice is the same for all $\eps$ small enough and for a fixed $y$.

Therefore in the case where $I_\eps(y)$ is not reduced to $\{0\}$, we have that either
\[
\begin{split}
&|\{t\,|\;f_i(t,y)\in E_\eps\cap \mathcal{C}\}|=(2\eps^2/\gamma)^{1/2k_0}+O(\eps^{2/k_0}),\quad \forall \eps<\bar \eps(y),\\
&\mbox{or}\ |\{t\,|\;f_i(t,y)\in E_\eps\cap \mathcal{C}\}|=2\,(2\eps^2/\gamma)^{1/2k_0}+O(\eps^{2/k_0}),\quad \forall \eps<\bar \eps(y).
\end{split}
\]
In both cases, for any $\sigma'<\kappa^{1/d}$ and therefore $\sigma'<\kappa^{1/k_0}$ , then provided $\eps\leq \bar\eps(y)$ and $\eps\leq (\kappa^{1/k_0}-\sigma)^{k_0}/C$ for some $C$,
\begin{equation}
|\{t\,|\;f_i(t,y)\in E_{\kappa\,\eps}\cap \mathcal{C}\}|\geq \sigma\,|\{t\,|\;f_i(t,y)\in E_\eps\cap \mathcal{C}\}|.\label{boundsegment}
\end{equation}
Applying the third point in \eqref{nondegeneratecurvature} and the assumption on $\mu$, we have that
\[
\begin{split}
  &\mu(E_\eps\cap \mathcal{C})=C\,\mathcal{H}^{d_0}|_M(E_\eps\cap \mathcal{C}) = C\,\int_{[0,\ 1]^{d_0-1}} |\{t\,|\; f_i(t,y)\in O\}|\,\nu_i(dy),\\
  &\mu(E_{\kappa\,\eps}\cap \mathcal{C})=C\,\mathcal{H}^{d_0}|_M(E_{\kappa\,\eps}\cap \mathcal{C}) = C\,\int_{[0,\ 1]^{d_0-1}} |\{t\,|\; f_i(t,y)\in O\}|\,\nu_i(dy).
\end{split}
\]
We just have to be careful about the condition $\eps\leq \bar \eps(y)$ when applying \eqref{boundsegment} to the equalities above as $\bar\eps$ depends (continuously) on $y$. By the dominated convergence theorem and \eqref{boundsegment}, we trivially have that
\[
\liminf_{\eps\to 0}\frac{\mu(E_\eps\cap \mathcal{C})}{\mu(E_{\kappa\,\eps}\cap \mathcal{C})}>\sigma',
\]
and therefore for any $\sigma<\sigma'$, there exists some $\eps_0$ s.t. if $\eps<\eps_0$ then
\[
\mu(E_\eps\cap \mathcal{C})\geq \sigma\;\mu(E_{\kappa\,\eps}\cap \mathcal{C}),
\]
finishing the proof.
        \end{proof}
        
We denote by $L:\;\R^N\to\R^d$ any non-linear function that is a combination of a shift, linear operation and as a non-linear function the absolute value; namely
\begin{equation}
L(x)_i=\left|\sum_{j=1}^N M_{ij} (x_j+s_j)\right|,\label{layerfunction}
  \end{equation}
where $s\in\R^N$ is the shift and $M\in M_{N,d}(\R)$ is a matrix.

We then have Theorem \ref{thm:data_clusters_to_confidences} as a consequence of 
\begin{theorem}\label{thm:propogation_of_cond_B}
Assume that the measure $\mu$ satisfies \eqref{doublingstrip} and $L$ is given by \eqref{layerfunction}. Then the pushforward $L_\#\mu$ also satisfies \eqref{doublingstrip} though with the new constants $m_0',\,\delta',\,\sigma',\;\kappa'$ but the same $k$ and $r$.
\end{theorem}
We recall that $L_\#\mu$ is defined by $L_\#\mu(O)=\mu(L^{-1}(O))$. 

\begin{remark}\label{rmk:propagation_of_condition_B}
	Theorem \ref{thm:propogation_of_cond_B} allows us to propagate condition B backwards. That is, we transfer condition B on the values taken by the last layer of a DNN before softmax to a similar condition on the second to last layer, then the third to last layer, until we reach a condition on the training set.
\end{remark}

To prove Theorem \ref{thm:propogation_of_cond_B}, we decompose $L$ into a linear part and the absolute value with propositions on each.
\begin{proposition}
Assume that the measure $\mu$ satisfies \eqref{doublingstrip} and that $M\in M_{N,d}(\R)$. Then the pushforward $M_\#\mu$ also satisfies \eqref{doublingstrip} with the same constants.\label{propdoublingstrip}
  \end{proposition}
\begin{proof}
  We simply observe that if $v_i\cdot x\leq t_i$ then any $y$ s.t. $M\,y=x$ also satisfies that
  \[
(M^T\,v_i)\cdot y\leq t_i. 
\]
Similarly if $(x-s)^T\,Q\,(x-s){\leq \eps}$ and $s=M\,s'$ then any $y$ s.t. $M\,y=x$ satisfies that
\[
(y-s')^T\,(M^T\,Q\,M)\,(y-s'){\leq \eps}.
\]
Of course $M^T\,Q\,M$ is automatically symmetric and positive semi-definite (for any $y$, $y^T\,M^T\,Q\,M\,y=(M\,y)^T\,Q\,(M\,y)\geq 0$). Moreover $\mbox{rank}\,(M^T\,Q\,M)\leq \mbox{rank}\,Q\leq r$. 
 Hence
\[
\begin{split}
  &M_\# \mu\left(\eps\,E(Q,s,v_1,t_1,\ldots,v_k,t_k)\right)=M_\# \mu\left(\{x,\;(x-s)^T\,Q\,(x-s)\leq \eps,\ v_i\cdot x\leq t_i, i=1\dots k\}\right)\\
  \  &=\mu\left(\{y,\;(y-s')^T\,(M^T\,Q\,M)\,(y-s')\leq \eps,\ (M^T\,v_i)\cdot x\leq t_i, i=1\dots k\}\right)=\mu\left(\eps\,E(M^T\,Q\,M,s',M^T\,v_1,t_1,\ldots,M^T\,v_k,t_k) \right).\\
\end{split}
\]
And similarly
\[
M_\# \mu\left(\kappa\,\eps\,E(Q,s,v_1,t_1,\ldots,v_k,t_k)\right)= \mu\left(\kappa\,\eps\,E(M^T\,Q\,M,s',M^T\,v_1,t_1,\ldots,M^T\,v_k,t_k) \right). 
\]
Since \eqref{doublingstrip} holds on $\mu$ for $E(M^T\,Q\,M,s',M^T\,v_1,t_1,\ldots,M^T\,v_k,t_k)$ for $\eps<\eps_0$, it trivially holds on $M_\#\,\mu$ on $E(Q,s,v_1,t_1,\ldots,v_k,t_k)$ for the same $k$, $m_0$, $\delta$, $\kappa,\;\sigma$ and even $\eps<\eps_0$. 
\end{proof}

The second and last part consists in handling the absolute value with
\begin{proposition}
Denote by $A:\,\R^N\to\R^N$ the absolute value function $A(x)=(|x_1|,x_2\ldots,x_N)$. Assume that $\mu$ solves \eqref{doublingstrip} for some $k$, $\kappa$, $\sigma$, $m_0$ and $\delta$.  Then $A_\#\,\mu$ solves \eqref{doublingstrip} wit

h the new constants $k'=k-1$, $m_0'=2\,m_0$, $\sigma'=\sigma/2$ but the same $r,\;\kappa,\;\delta$. Furthermore if $\mu(\{x_1=0\}=0)$ then we may take $\sigma'=\sigma$.
\label{propabsolute}
\end{proposition}
\begin{proof}
  Consider any truncated ellipsoidal cylinders 
  \[
{\varepsilon}\,E=\{(x-s)^T\,Q\,(x-s)\leq \eps, \ v_i\cdot x\leq t_i, i=1\dots k-1\}.
\]
Define the new truncated cylinders
 \[
\eps\,\bar E=\{(x-s)^T\,Q\,(x-s)\leq \eps, \ v_i\cdot x\leq t_i, i=1\dots k-1,\;x_1\geq 0\},
\]
which consists in adding $v_{k}=(-1,0,\ldots,0)$ and $t_{k}=0$ in the definition of $E$.

  The inverse image $A^{-1}(\eps\,E)$ consists of $\eps\,\bar E$ and of
  \[
  \begin{split}
    &\eps\,\bar E'=\{(x-s)^T\,Q'\,(x-s)\leq \eps, \ v_i'\cdot x\leq t_i, i=1\dots k\},\quad v_i'=(-v_{i,1},v_{i,2},\ldots,v_{i,N}),\\
    &Q'_{1j}=Q'_{j1}=-Q_{1j}\ j\neq 1,\quad Q'_{ij}=Q_{ij}\ \mbox{otherwise}.
\end{split}
  \]
  Observe that $Q$ and $Q'$ have the same rank. Denote by $T$ the matrix corresponding to the reflection for the hyperplane $x_1=0$ ($T_{11}=-1$, $T_{ii}=1$ for $i>1$ and $T_{ij}=0$ if $i\neq j$). Then one has that $Q'=T\,Q\,T$ as $Q'_{ij}=\sum_{k,l} T_{ik}\,Q_{kl}\,T_{lj}=T_{ii}\,Q_{ij}\,T_{jj}$. Since $T$ is a bijection then $\mbox{rank}\,Q=\mbox{rank}\,Q'$.
  
For any $\kappa$, we have of course that $A^{-1}(\kappa\,\eps\,E)=\kappa\,\eps\,\bar E\cup\kappa\,\eps\,\bar E'$ and typically we want to apply \eqref{doublingstrip} to those two truncated cylinders.

Notice that for any $\eps$, $\eps\bar E\subset \{x_1\geq 0\}$ while $\eps\,\bar E'\subset \{x_1\leq 0\}$. It then makes sense to separate
\[\begin{split}
&\eps\,\bar E=E^+_{\eps}\cup E^0_{\eps},\quad \eps\,\bar E'=E^-_{\eps}\cup E^0_{\eps},\\
& E^+_\eps=\eps\,\bar E\cap \{x_1>0\},\quad E^0_{\eps}=\eps\,\bar E\cap \{x_1=0\}=\eps\,\bar E'\cap \{x_1=0\},\quad E^{{-}}_\eps=\eps\,\bar E'\cap \{x_1<0\}.
  \end{split}
  \]
Now applying \eqref{doublingstrip} to $\bar E$, we find $\eps_0$ s.t. $\forall\eps<\eps_0$
\begin{equation}
\mu(\kappa\,\eps\,\bar E)=\mu(E^+_{\kappa\,\eps})+\mu(E^0_{\kappa\,\eps})\geq \min\left\{\delta,\;\max\{m_0,(1+\sigma)\,\mu(\eps\,\bar E)\}-m_0\right\}. 
\label{measE+}
\end{equation}
Similarly applying \eqref{doublingstrip} to $\bar E'$, there exists $\eps_0'$ s.t. $\forall \eps<\eps_0'$
\begin{equation}
\mu(\kappa\,\eps\,\bar E')=\mu(E^-_{\kappa\,\eps})+\mu(E^0_{\kappa\,\eps})\geq \min\left\{\delta,\;\max\{m_0,(1+\sigma)\,\mu(\eps\,\bar E')\}-m_0\right\}. 
\label{measE-}
\end{equation}
We of course take now any $\eps<\min(\eps_0,\eps_0')$ and if either $(1+\sigma)\,\mu(\eps\,\bar E)\geq \delta$ or $(1+\sigma)\,\mu(\eps\,\bar E')\geq\delta$, we are done.

Assume now that $(1+\sigma)\,\mu(\eps\,\bar E)<m_0$ or $(1+\sigma)\,\mu(\eps\,\bar E')<m_0$. If both hold then again the proof is done, otherwise take for example $(1+\sigma)\,\mu(\eps\,\bar E')<m_0$ with $(1+\sigma)\,\mu(\eps\,\bar E)\geq m_0$. Then
\[
(1+\sigma)\,\mu(A^{-1}(\eps\,E))\leq (1+\sigma)\,\mu(\eps\,\bar E)+m_0>m_0,
\]
and
\[
\mu(A^{-1}(\kappa\,\eps\,E))\geq (1+\sigma)\,\mu(\eps\,\bar E)-m_0\geq (1+\sigma)\,\mu(A^{-1}(\eps\,E))-2\,m_0,
\]
satisfying \eqref{doublingstrip} with $m_0'=2\,m_0$.

Finally there only remains the case with both $m_0\leq (1+\sigma)\,\mu(\eps\,\bar E)<\delta$ and $m_0\leq (1+\sigma)\,\mu(\eps\,\bar E')<\delta$. Then
\[
\begin{split}
  \mu(A^{-1}(\kappa\,\eps\,E))&=\mu(E^+_{\kappa\,\eps})+\mu(E^0_{\kappa\,\eps}) +\mu(E^-_{\kappa\,\eps})\\
  &\geq \mu(E^+_{\eps})+\mu(E^0_{\eps}) +\mu(E^-_{\eps})+\frac{\mu(E^+_{\kappa\,\eps})+2\mu(E^0_{\kappa\,\eps}) +\mu(E^-_{\kappa\,\eps})- \mu(E^+_{\eps})-2\,\mu(E^0_{\eps}) -\mu(E^-_{\eps})}{2}.  
\end{split}\]
Applying now \eqref{measE+} and \eqref{measE-}, we find that
\[
\mu(A^{-1}(\kappa\,\eps\,E))\geq \mu(A^{-1}(\eps\,E)) +\frac{\sigma}{2}\,(\mu(\eps\,\bar E) +\mu(\eps\,\bar E')) -m_0 \geq (1+\sigma/2)\, \mu(A^{-1}(\eps\,E)) -m_0,
\]
proving \eqref{doublingstrip} with $\sigma'=\sigma/2$. Of course if $\mu(\{x_1=0\})=0$ then $\mu(E^0_{\kappa\,\eps})=0$ and we can take $\sigma'=\sigma$.
\end{proof}

\section{Appendix}\label{sec:appendix}
	\subsection{Derivations of Conditions A and B}\label{sec:condition_derivations}
	
	The goal co conditions A and B is to ensure that $\{\delta X(s,\alpha(t_0)):s\in G^C_\eta(t_0)\}$ does not concentrate near its minimum $\beta$. Conditions A and B accomplish this goal in different ways. 
	
	\subsubsection{Condition A}\label{sec:condition_A_explanation} If there is no concentration of $\delta X(s,\alpha(t_0))$ near $\beta$, we expect that for a small interval whose left endpoint is $\beta$, more $\delta X$ values are in the right half of this interval than in the left. In other words, for small $a>0$,
	\begin{equation}\label{eq:tdsm_1}
	\nu\left(\left\{s\in T: \beta\leq \delta X(s,\alpha(t_0))<\beta+a \right\}\right)<\nu\left(\left\{s\in T: \beta+a\leq\delta X(s,\alpha(t_0))<\beta+2a \right\}\right).
	\end{equation}
	Since $\beta=\min_{s\in T}\delta X(s,\alpha(t_0))$, we may write \eqref{eq:tdsm_1} equivalently as
	\begin{equation}\label{eq:tdsm_2}
	\nu\left(\left\{s\in T: \delta X(s,\alpha(t_0))<\beta+a \right\}\right)<\frac{1}{2}\nu\left(\left\{s\in T: \delta X(s,\alpha(t_0))<\beta+2a \right\}\right).
	\end{equation}
	Now replace $1/2$ with a continuous parameter $\Lambda $:
	\begin{equation}\label{eq:tdsm_3}
	\nu\left(\left\{s\in T: \delta X(s,\alpha(t_0))<\beta+a \right\}\right)<\Lambda \nu\left(\left\{s\in T: \delta X(s,\alpha(t_0))<\beta+2a \right\}\right).
	\end{equation}
	Since all masses are less than $1$, we can control concentration near $\beta$ better by increasing the exponent on the right side of \eqref{eq:tdsm_3}:
	\begin{equation}\label{eq:tdsm_4}
	\nu\left(\left\{s\in T: \delta X(s,\alpha(t_0))<\beta+a \right\}\right)<\Lambda \nu\left(\left\{s\in T: \delta X(s,\alpha(t_0))<\beta+2a \right\}\right)^{\psi+1},\quad \psi>0.
	\end{equation}
	By letting $x_1=\beta$ and $x_2=\beta+2a$, we may write this condition as
	\begin{equation}\label{eq:tdsm_5}
	\nu\left(\left\{s\in T: \delta X(s,\alpha(t_0))<\frac{x_1+x_2}{2} \right\}\right)<\Lambda \nu\left(\left\{s\in T: \delta X(s,\alpha(t_0))<x_2 \right\}\right)^{\psi+1},\quad \psi>0.
	\end{equation}
	We will see in the proof of Lemma \ref{lem:loss_bound_1} that \eqref{eq:tdsm_5} leads to an excellent bound on $\bar L(t_0)$.
	
	The inequality \eqref{eq:good_set_and_loss} provides a lower bound for $G_\ell(t)$ via $\bar L(t_0)$ for $\ell>0$, but by adjusting condition A slightly, we can improve this inequality. In particular, we would like to apply a condition like \eqref{eq:tdsm_5} with $x_1=0$ and $x_2=2\ell$ to obtain a bound on $\nu(G_\ell^C(t))=\nu(\{s\in T:\delta X(s,\alpha(t))<\ell\})$:
	\begin{equation}\label{eq:tdsm_6}
	1-\nu(G_\ell(t))=\nu(\{s\in T:\delta X(s,\alpha(t))<\ell\})\leq \Lambda \nu(\{s\in T:\delta X(s,\alpha(t))<2\ell\})^{\psi+1}=\Lambda (1-\nu(G_{2\ell}))^{\psi+1}
	\end{equation}
	Assuming that the map $\ell\mapsto\nu(G_\ell(t))$ is continuous at $\ell=0$ (i.e., there is no $s\in T$ with $\delta X(s,\alpha(t))=0$), then in the limit as $\ell\to 0$:
	\begin{equation}\label{eq:tdsm_7}
	1-\nu(G_0(t))\leq \Lambda (1-\nu(G_0(t)))^{\psi+1}.
	\end{equation}
	Dividing both sides of \eqref{eq:tdsm_7} by $1-\nu(G_0(t))$ and using \eqref{eq:loss_and_accuracy}, we obtain
	\begin{equation}\label{eq:tdsm_8}
	1\leq \Lambda (1-\nu(G_0(t)))^\psi\leq \Lambda \left(\frac{\bar L(t_0)}{\log 2}\right)^\psi.
	\end{equation}
	If $\bar L(t_0)$ is small, \eqref{eq:tdsm_8} may not be satisfied. We do not want to exclude distributions $\{\delta X(s,\alpha(t_0)):s\in T\}$ with small $\bar L(t_0)$, so \eqref{eq:tdsm_5} is insufficient. A simple solution to this problem is to add a new term depending on $x_1$ to the right side of \eqref{eq:tdsm_5}. The new term must vanish when $x_1=\beta$ so that the loss bound obtained from \eqref{eq:tdsm_5} still holds. The new term must also not exclude distributions with small $\bar L(t_0)$ when $x_1=0$. The obvious candidate is $\nu(\{s\in T:\delta X(s,\alpha(t))<x_1)\})^\phi$ for some power $\phi>0$, so the TDSM condition becomes
	\begin{equation}\label{eq:tdsm_9}
	\begin{split}
	\nu\left(\left\{s\in T: \delta X(s,\alpha(t_0))<\frac{x_1+x_2}{2} \right\}\right)&<\Lambda \nu\left(\left\{s\in T: \delta X(s,\alpha(t_0))<x_2 \right\}\right)^{\psi+1}\\
	&+\Lambda \nu\left(\left\{s\in T: \delta X(s,\alpha(t_0))<x_1 \right\}\right)^{\phi},\quad \psi,\phi>0.
	\end{split}
	\end{equation}
	Applying the analysis used to obtain \eqref{eq:tdsm_8} to \eqref{eq:tdsm_9}, we get
	\begin{equation}
	1\leq \Lambda \left(\frac{\bar L(t_0)}{\log 2}\right)^\psi+\Lambda \left(\frac{\bar L(t_0)}{\log 2}\right)^{\phi-1}.
	\end{equation}
	This is satisfied trivially as long as $\Lambda \geq 1$ and $0<\phi<1$.
	
	\subsubsection{Condition B} Consider a small interval $J=[\beta-x,\beta+x]$ for $x>0$. There may be some $\delta X(s,\alpha(t_0))$ values in the right half of $I$, but if the $\delta X(s,\alpha(t_0))$ values do not cluster near $\beta$, then there should be more $\delta X(s,\alpha(t_0))$ values to the right of $J$. In other words, If we increase the width of $J$ from $2x$ to $2\kappa x$ for some $\kappa>1$, leaving its center in place, the number of $\delta X(s,\alpha(t_0))$ values it contains should increase:
	\begin{equation}\label{eq:condB_development_1}
	\nu\left(\{s\in T: \delta X(s,\alpha(t_0))\in[\beta-\kappa x,\beta+\kappa x]\}\right)>\nu\left(\{s\in T:\delta X(s,\alpha(t_0))\in[\beta-x,\beta+x]\}\right).
	\end{equation}
	Now let $I$ be any interval. Denoting by $\kappa I$ the interval whose center is the same as $I$ but whose width is increased by a factor of $\kappa$, \eqref{eq:condB_development_1} becomes 
	\begin{equation}\label{eq:condB_development_2}
	\nu\left(\{s\in T: \delta X(s,\alpha(t_0))\in\kappa I\}\right)>\nu\left(\{s\in T:\delta X(s,\alpha(t_0))\in I\}\right).
	\end{equation}
	We can strengthen \eqref{eq:condB_development_2} by introducing a factor of $(1+\sigma)$ on the right hand side, where $\sigma>0$:
	\begin{equation}\label{eq:condB_development_3}
	\nu\left(\{s\in T: \delta X(s,\alpha(t_0))\in\kappa I\}\right)>(1+\sigma)\nu\left(\{s\in T:\delta X(s,\alpha(t_0))\in I\}\right).
	\end{equation}
	Of course $T$ is a finite set, so one can take $I$ arbitrarily small containing a single element, in which case $\kappa I$ may still contain only that same element. Therefore, it is necessary to introduce $m_0$, a small mass which accounts for when $I$ is so small that \eqref{eq:condB_development_3} may not hold. When the left right hand side of \eqref{eq:condB_development_3} is too small, \eqref{eq:condB_development_3} need not hold:
	\begin{equation}\label{eq:condB_development_4}
	\nu\left(\{s\in T: \delta X(s,\alpha(t_0))\in\kappa I\}\right)>\max\left\{m_0,(1+\sigma)\nu\left(\{s\in T:\delta X(s,\alpha(t_0))\in I\}\right)\right\}.
	\end{equation}
	Finally, If $I$ is too large, it is to be expected that increasing its width does not increase its mass much, e.g., if $I$ contains all $\delta X(s,\alpha(t_0))$ values. Moreover, we will use condition B to control the distribution of only a small number of misclassified objects. Therefore, if we introduce $\delta$ which is the maximum mass of intervals we will consider with condition B:
	\begin{equation}
	\nu\left(\{s\in T: \delta X(s,\alpha(t_0))\in\kappa I\}\right)>\min\left\{\delta,\max\left\{m_0,(1+\sigma)\nu\left(\{s\in T:\delta X(s,\alpha(t_0))\in I\}\right)\right\}\right\}.
	\end{equation}
	This completes the derivation of condition B.
	\subsection{Proof of Lemma \ref{lem:double_exponential_sum_constant}}\label{sec:technical_lemma}
	
	Here we present the proof of Lemma \ref{lem:double_exponential_sum_constant}
	
	\begin{proof}
		For simplicity, write $\varepsilon K=x$. Since $0<\varepsilon<1/(2K)$, we consider $0<x<1/2$. First, observe that $$\sum_{k\geq 0} e^{2^k log(x)-\frac{\eta+1}{2^{k+1}}}-C e^{-\sqrt{2|\log(x)|(\eta+1)}}\leq 0$$ if and only if $$C\geq  \sum_{k\geq 0} e^{2^k log(x)-\frac{\eta+1}{2^{k+1}}+\sqrt{2|\log(x)|(\eta+1)}}.$$ To simplify the problem, let $y=\sqrt{-\log(x)}$ and $z=\frac{1}{y}\sqrt{\frac{\eta+1}{2}}$. Then $$2^k \log(x)-\frac{\eta+1}{2^{k+1}}+\sqrt{2|\log(x)|(\eta+1)}=-2^{-k}(z-2^k)^2y^2.$$ Since $0<x<1/2$ and $\eta>1$, we have $y>\sqrt{\log 2}$ and $z>\frac{1}{y\sqrt{2}}$ or equivalently, $y\geq\max\{\sqrt{\log 2},1/(z\sqrt{2})\}$. The problem is reduced to finding $C$ such that $$C\geq h(y,z):=\sum_{k\geq 0}e^{-2^{-k}(z-2^k)^2y^2}$$ for all $y\geq \max\{\sqrt{\log 2},1/(z\sqrt{2})\}$. We proceed to find $$C_0=\max_{y\geq\max\{\sqrt{\log 2},1/(z\sqrt{2})\},\;z>0}h(y,z)$$ so that it is always possible choose $C\leq C_0$. 
		
		It is easy to calculate $$\frac{\partial}{\partial y}h(y,z)=-2y\sum_{k\geq 0}2^{-k}\left(2^k-z\right)^2 e^{-2^{-k} \left(z-2^k\right)^2y^2}<0$$ for $y,z>0$. Thus, for fixed $z=z'$, $$\max_{y\geq \max\{\sqrt{\log 2},1/(z'\sqrt{2})\}}h(y,z')=h(\max\{\sqrt{\log 2},1/(z'\sqrt{2})\},z').$$ Therefore, $C_0=\max\{a,b\}$ where $$a=\max_{z\geq1/\sqrt{2\log(2)}}h(\sqrt{\log 2},z),\quad\text{and}\quad b=\max_{0<z\leq 1/\sqrt{2\log 2}} h(1/(z\sqrt{2}),z).$$ Therefore, we will estimate $a$ and $b$. 
		
		First we will estimate $a$. Let $$p_k(z)=2^{-2^{-k}(2^k-z)^2}$$ so that $h(\sqrt{\log 2},z)=\sum_{k\geq 0}p_k(z)$. Note $0\leq p_k(z)\leq 1$ for all $z\in\mathbb R$ and $k\geq 0$. Observe also that $(2^k-z)^2$ is a convex function, so it is bounded below by any tangent line. In particular, $$(2^k-z)^2\geq 2^k\left(\frac{3}{4}2^k-z\right)$$ and $$(2^k-z)^2\geq 2^{k+1}\left(z-\frac{3}{4}2^{k+1}\right).$$ From these, we obtain two upper bounds on $p_k(z)$: $$p_k(z)\leq \frac{2^k}{z}2^{-\left(\frac{3}{4}2^k-z\right)}$$ and $$p_k(z)\leq\frac{2^k}{z}2^{-2\left(z-\frac{3}{4}2^{k+1}\right)}.$$ The former is useful for $z\leq 2^{k-1}$. The latter is useful for $z\geq 2^{k+1}$. Recall also the following identities: $$\sum_{k=1}^n a^k=\frac{a \left(a^n-1\right)}{a-1}$$ and $$\sum_{k=n}^\infty a^k=\frac{a^n}{1-a}\quad\text{for $a<1$}.$$
		
		First consider $z<2$. We may write
		\begin{align*}
		h(\sqrt{\log 2},z)&=p_0(z)+p_1(z)+\sum_{k=2}^\infty p_k(z)\\
		&\leq 2+\sum_{k=2}^\infty 2^{-\left(\frac{3}{4} 2^k-z\right)}\\
		&\leq 2+4\sum_{k=2}^\infty 2^{-\frac{3}{4}2^k}\\
		&\leq 2+4\sum_{k=4}^\infty 2^{-\frac{3}{4} k}\\
		&=2+4\frac{2^{-\frac{3}{4}\times 4}}{1-2^{-3/4}}\\
		&\leq 2+\frac{5}{4}\\
		&=3.25
		\end{align*}  
		
		Now, suppose $z\geq 2$. Then there is an integer $m\geq 2$ so that $2^{m-1}\leq z<2^{m}$. We write $$h(\sqrt{\log 2},z)=p_{m-1}(z)+p_m(z)+\sum_{k=0}^{m-2}p_k(z)+\sum_{k=m+1}^\infty p_k(z),$$ and estimate both sums:
		\begin{multicols}{2}
			\begin{align*}
			\sum_{k=0}^{m-2}p_k(z)&\leq\sum_{k=0}^{m-2} 2^{-2\left(z-\frac{3}{4}2^{k+1}\right)}\\
			&\leq 2^{-2^m}\sum_{k=0}^{m-2}2^{\frac{3}{2}2^{k+1}}\\
			&\leq 2^{-2^m}\sum_{k=1}^{2^{m-2}}8^{k}\\
			&=2^{-2^m}\frac{8(8^{2^{m-2}}-1)}{7}\\
			&=\frac{8}{7}2^{-2^{m-2}}\\
			&\leq \frac{4}{7},
			\end{align*}
			
			\begin{align*}
			\sum_{k=m+1}^\infty p_k(z)&\leq \sum_{k=m+1}^\infty 2^{-\left(\frac{3}{4} 2^k-z\right)}\\
			&\leq 2^{2^m}\sum_{k=m+1}^\infty 2^{-\frac{3}{4}2^k}\\
			&\leq 2^{2^m}\sum_{k=2^{m+1}}^\infty 2^{-\frac{3}{4} k}\\
			&= 2^{2^m}\frac{2^{-\frac{3}{4}2^{m+1}}}{1-2^{-3/4}}\\
			&=\frac{2^{-2^{m-1}}}{1-2^{-3/4}}\\
			&\leq \frac{5}{8}.
			\end{align*}
		\end{multicols}
		\noindent Since clearly $p_{m-1}(z)$ and $p_m(z)$ are less than $1$, we have $$h(\sqrt{\log 2},z)\leq 2+\frac{4}{7}+\frac{5}{8}\approx 3.196$$ for $z\geq 2$. Thus, $a\leq 3.25$.
		
		Now we will calculate $b=\max_{0<z\leq 1/\sqrt{2\log 2}} h(1/(z\sqrt{2}),z)$. Let $$q_k(z)= e^{-2^{-k-1} \left(1-\frac{2^k}{z}\right)^2}.$$ Then $$q'_k(z)=\frac{e^{-2^{-k-1} \left(\frac{2^k}{z}-1\right)^2} \left(\frac{2^k}{z}-1\right)}{z^2}.$$ The sign of $q'_k(z)$ is the sign of $(2^k/z-1)$. For $0\leq z\leq 1/\sqrt{2\log 2}<1$, $(2^k/z-1)>0$, so $q_k(z)$ is increasing for all $z\in(0,1)$ and $k\geq 0$. Therefore,
		\allowdisplaybreaks
		\begin{align*}
		b&=h\left(1/\sqrt{2},1\right)\\
		&=\sum_{k\geq 0}e^{-2^{-k-1}(2^{k}-1)^2}\\
		&=\sum_{k\geq 0}e^{-2^{-k-1}(2^{2k}-2^{k+1}+1)}\\
		&=\sum_{k\geq 0}e^{-\frac{1}{2}(2^k-2+2^{-k})}\\
		&=\sum_{k\geq 0}e^{-\frac{1}{2}(2^{k/2}-2^{-k/2})^2}\\
		&=e^0+e^{-\frac12(2^{1/2}-2^{-1/2})^2}+\sum_{k\geq 2}e^{-\frac{1}{2}(2^{k/2}-2^{-k/2})^2}\\
		&\leq \frac{9}{5}+\sum_{k\geq 2}e^{-2^{k-2}}\\
		&\leq \frac{9}{5}+\sum_{k\geq 1}e^{-k}\\
		&=\frac{9}{5}+\frac{1}{e-1}\\
		&< \frac{5}{2}.
		\end{align*}
		
		We conclude that $C_0=\max\{a,b\}\leq 3.25$
	\end{proof}



\end{document}